\documentclass[12pt]{book}

\usepackage[utf8]{inputenc}
\usepackage[margin=1.25in]{geometry}
\usepackage{amsmath, amssymb, amsthm}
\usepackage[english]{babel}
\usepackage{blindtext}
\usepackage{enumerate}
\usepackage{mathtools} 
\usepackage[backend=biber]{biblatex}
\usepackage[nottoc,notlot,notlof]{tocbibind}
\usepackage[titletoc]{appendix}

\usepackage{color}
\usepackage{hyperref}
\usepackage{tocloft}
\hypersetup{
    colorlinks=true, 
    linktoc=all,     
    linkcolor=blue,  
    citecolor= blue  
}

\usepackage{fancyhdr}
\usepackage{apptools}
\AtAppendix{\counterwithin{lemma}{section}}
\AtAppendix{\counterwithin{definition}{section}}
\AtAppendix{\counterwithin{theorem}{section}}
\AtAppendix{\counterwithin{proposition}{section}}

\newtheorem{theorem}{Theorem}[section]
\newtheorem{corollary}[theorem]{Corollary}
\newtheorem{lemma}[theorem]{Lemma}
\newtheorem{proposition}[theorem]{Proposition}
\newtheorem{definition}[theorem]{Definition}
\newtheorem*{remark}{Remark}

\theoremstyle{definition}
\newtheorem{example}{Example}[section]

\newcommand{\labs}{\displaystyle\left\lvert}
\newcommand{\rabs}{\right\rvert}

\newcommand{\dharm}{\mathcal{D}_{\textup{harm}}}

\newcommand{\blabs}{\displaystyle\Biggl\lvert}
\newcommand{\brabs}{\Biggr\rvert}

\title{On the Cauchy Integral and Jump Decomposition}
\author{James A. C. Young}
\date{Summer 2021}

\begin{document}

\pagenumbering{roman}

\maketitle

\chapter*{Introduction}
\addcontentsline{toc}{chapter}{Introduction}

\section*{Preface}

The purpose of these notes is to introduce the reader to the Cauchy type integral and the Sokhotski formulae (also know as the Jump decomposition, or simply the Jump problem). The article is split into two main chapters. The first is based on the contents of the first few sections of the book ``Boundary Value Problems" by F.D. Gakhov \cite{gakhov2014boundary}. The style of presentation is geared towards readers who have a background at the level of a first course in complex variables, and are comfortable with basic arguments in analysis. One of the goals of the text is to be expository, and for that reason nearly all proofs will contain thoroughly worked out arguments and explanations. After establishing the basic ideas in the first chapter, we will skip past much rigorous training to (informally) discuss some open questions and advancements in the field lying on the fringe of our understanding in the second chapter. We will see how many disciplines work together to tell a richer story, some of which include geometric function theory, singular integral operators, and boundary value problems.

\medskip

A very special thank you to Eric Schippers for the guidance and tremendous amount of assistance over the summer months of 2021 that made this presentation possible.

\begin {flushright}
  James A. C. Young, September 7th 2021.
\end {flushright}
\section*{Notation}

Throughout the main article, the general notation for contours will be $C$. Unless otherwise stated, $C$ may be open or closed. The contours we will be working with are \emph{well-behaved}, which will be given an adequate definition shortly. This is done to keep attention directed at the main formulation of the article: The Jump formula. As a preview to the research problem, exploring similar ideas for much more exotic and ``poorly-behaved" curves is part of the current work of analysts in the field.

\smallskip
We say that a contour $C$ is \emph{smooth} if it satisfies the following properties:

\begin{enumerate}
    \item The contour is either a simple contour, meaning it does not self intersect, or a simple closed contour, meaning it only intersects at its end points.
    \item It is of class $\mathcal{C}^1$, meaning there is a parameterization of the contour that is continuously differentiable. Furthermore, we will assume the default parameterization of contours to be taken in the counterclockwise sense.
    \item Tangent vectors to the curve never vanish, meaning the curve has no cusps. This can also be stated as the derivative of the parameterization is never zero. An example of a curve with a cusp is the curve described by the equation $y^3=x^2$ in the $xy$-plane (The origin being the location of the cusp).
\end{enumerate}

A contour $C$ is \emph{piecewise} smooth if it consists of finitely-many smooth contours joined end-to-end. For example, a right triangle in the plane with vertices $(0,0)$, $(1,0)$, and $(0,1)$ is a piecewise smooth contour. For simplicity, smooth contours will always fall under the assumption of possibly being piecewise smooth. \emph{Endpoints} of a contour include both the traditional endpoints of an open contour which we will call \emph{true} endpoints (which will always be listed with respect to direction of traversal), as well as any endpoint of a particular contour included in the definition of a piecewise smooth contour where the two smooth contours that meet at this point couldn't together be considered smooth. In the case of our right triangle example, we would be counting the vertices as endpoints since they form ``corners."

\smallskip

Finally, our contours will be \emph{rectifiable}, meaning they have finite length (and consequently are bounded). This is actually implied by our assumption of a continuously differentiable parameterization (see chapter 6 in \cite{apostol1957mathematical}), but it is worth explicitly mentioning this property for the sake of clarity. The length of a contour $C$ is denoted by $\ell (C)$. All of these properties taken together encapsulate the notion of well-behaved, smooth curves we will adhere to.

\medskip

Let $C$ be a simple closed contour in the complex plane. Recall that the Jordan curve theorem states that such a curve divides the plane into two connected components; one bounded and one unbounded. See Newman's book \cite{newman1939elements} for a proof of Jordan's Theorem. The bounded component of the complement of $C$ is called the \emph{interior} of $C$ and will be denoted by $D^{+}$, while the unbounded component of the complement is called the \emph{exterior} of $C$ and will be denoted by $D^{-}$. By our choice of parameterization direction, $D^{+}$ will be the region to the left of direction of travel, and $D^{-}$ will be the region to the right of direction of travel.

\medskip

As a rule, we generally represent complex variables with $z$, $w$, or $\tau$, and real variables with $x$ or $t$. As a means of distinction, we use the symbol $\log$ for the complex logarithm, and $\ln$ for the real logarithm. Derivatives of functions will either be denoted by attaching an apostrophe to the function (if the variable in which differentiation is with respect to is clear from context), or by the function symbol with a subscript that labels the particular variable. For example, the derivative of $f(z)$ may be written as $f'(z)$, while the derivative of $g(z,\tau)$ with respect to $z$ will be written as $g_z(z,\tau)$.

\cleardoublepage 
\tableofcontents
\cleardoublepage 

\pagenumbering{arabic}

\chapter{Cauchy Integrals and The Jump Problem}
\label{chapter1}

\section{The Cauchy Type Integral}

First, we will revisit the astonishing Cauchy integral formula. This formula is fundamental to the theory of functions of a complex variable. We wish to extend the formula to both a wider class of functions and a larger domain\textemdash that being the unbounded complement of the contour. To precisely define what is meant by the latter, we need to become acquainted with the point at infinity. We start with the following definition.

\begin{definition} [Holomorphic Extension To The Point at Infinity]
\label{definition_HolomorphicExtensionToInfinity}
Let $C$ be a bounded simple closed contour, and let $f(z)$ be a function that is analytic on $D^{-}$. If $f\left(1/z\right)$ has a removable singularity at $z = 0$, then we define the holomorphic extension of $f$ to the point at infinity as
\smallskip
\[
f(\infty) \coloneqq \lim_{z \to 0} f\left(\frac{1}{z}\right).
\]
\end{definition}
\smallskip

The point at infinity is crucial in the development of complex analysis. The plane together with the point at infinity is called the Riemann sphere, and is the canonical example of a Riemann surface\textemdash the general objects on which complex analysis is performed.

\smallskip

The wider class of functions for the extension includes functions who maintain continuity, but not necessarily analyticity on the contour. The need for such a generalization\textemdash beyond testing the bounds of the theory\textemdash will be uncovered only after the investigation of the main problem and its consequences.

\smallskip

Next, we gather some preliminary results.

\begin{proposition}
\label{pass lim to integral}
Let $(f_n)_{n=1}^{\infty}$ be a sequence of continuous functions that converge uniformly to a function $f$ on a smooth contour $C$. Then
\smallskip
\[
\lim_{n \to \infty} \int_C f_n (z) \,dz = \int_C f(z) \,dz .
\]

That is, we can pass the limit through the integral sign.
\end{proposition}

\begin{proof}
First, it is clear that the limit function is continuous, and hence integrable over $C$. Let $\varepsilon > 0$. Since $f_n$ converges to f uniformly on $C$, there is $N \in \mathbb{N}$ with the property that $\labs f_n - f\rabs < \varepsilon$ on $C$ whenever $n \geq N$. An application of the estimation lemma yields
\smallskip
\[
\labs \int_{C} f_n (z)\,dz - \int_{C} f(z)\,dz \rabs = \labs \int_{C} \left[f_n (z) - f(z)\right] \,dz \rabs < \varepsilon \cdot \ell(C)  
\]

\noindent
whenever $n \geq N$, which completes the proof.
\end{proof}

Interchanging limits and integrals is a common and important analytic technique. We will encounter the situation where we wish to pass a limit through an integral to a function that has a continuous parameter that acts as a sequence indexed continuously with an interval of the real line. The following lemma will accomplish this.

\begin{lemma}
\label{pass cont lim to integral}
Suppose that the real-valued function $f(x,t)$ is uniformly continuous on a set $S \times [a,b] \subset \mathbb{R}^2$. Then for any $x_0 \in S$:
\smallskip
\[
\lim_{x \to x_0} \int_a^b f(x,t) \,dt = \int_a^b f(x_0,t) \,dt .
\]

That is, the integral acts as a continuous function in $x$.
\end{lemma}

\begin{proof}
By uniform continuity of $f$ on $S\times[a,b]$, it follows that for any $\varepsilon > 0$, there is $\delta > 0$ with the property that $|f(x,t)-f(x_0,t)|<\varepsilon/(b-a)$ whenever $|x-x_0|<\delta$ and for every $t \in [a,b]$. Assuming $x\in S$ satisfies this requirement and estimating the difference of the integrals:
\begin{align*}
\labs \int_a^b f(x,t) \,dt - \int_a^b f(x_0,t) \,dt \rabs = \labs \int_a^b [f(x,t) - f(x_0,t)] \,dt \rabs &\leq \int_a^b \labs f(x,t) - f(x_0,t)\rabs \,dt \\
&< \int_a^b \frac{\varepsilon}{b-a}\,dt \\
&=\varepsilon
\end{align*}
\noindent
as required.
\end{proof}

The following is the same type of result, but pertaining to contour integrals in particular. It too will be extremely important, but through a different application.

\begin{corollary}
\label{pass cont lim to contour integral}
Suppose that the complex-valued function $f(\tau,z)$ is uniformly continuous on a set $C \times D \subset \mathbb{C}^2$, where $C$ is a smooth contour in the $\tau$-plane and $D$ is a domain in the $z$-plane. Then for any $z_0\in D$:
\smallskip
\[
\lim_{z \to z_0} \int_C f(\tau,z) \,d\tau = \int_C f(\tau,z_0) \,d\tau .
\]

\noindent
That is, the integral acts as a continuous function in $z$.
\end{corollary}

\begin{proof}
A nearly identical argument as the one used for Lemma \ref{pass cont lim to integral} does the trick.
\end{proof}
\smallskip

The next concept is a form of equivalence for curves embedded in a set called a \emph{homotopy}.

\begin{definition}[Homotopic Contours]
Let $C_1$ and $C_2$ be two simple closed contours parameterized by functions $\alpha_1$ and $\alpha_2$ (respectively) that both share the common domain $[a,b]$. Let $D \subset \mathbb{C}$ be a subset of the complex plane that contains both $C_1$ and $C_2$. Then $C_1$ and $C_2$ are said to be homotopic in $D$ if there exists a continuous function $h:[0,1] \times [a,b] \rightarrow D$ such that 

\begin{enumerate}[i]
    \item\!\!. $h(0,t) = \alpha_1 (t)$ \,\textup{if}\, $t \in [a,b]$
    \item\!\!. $h(1,t) = \alpha_2 (t)$ \,\textup{if}\, $t \in [a,b]$
    \item\!\!. $h(s,a) = h(s,b)$ \,\textup{for all}\, $s \in [0,1]$.
\end{enumerate}

The function h is called a homotopy.
\end{definition}

The interpretation of two contours being homotopic is that one can be continuously deformed into the other, without leaving the set $D$. For each fixed $s_0 \in (0,1)$, we can think of $h(s_0,t)$ as tracing out an intermediate curve $C_{s_0}$ as $t$ varies over $[a,b]$. The conditions on the function $h$ asserts that the intermediate contours are closed, and that the beginning and end states of the homotopy $h$ correspond to the homotopic contours $C_1$ and $C_2$, respectively.

\smallskip

Homotopic paths can also tell us something about the structure of the underlying set in which the curves are embedded. For instance, consider an annulus in the plane. Intuitively, it can be thought of as a disk with a concentric hole. Having a hole in the space distinguished the set from a regular disk in the homotopic sense. Notice that a contour surrounding the hole cannot be continuously deformed into a contour that doesn't surround the hole, since any deformation would either require the curve to pass through the hole, leaving the set, or be cut. If a set has no ``holes" in it\textemdash or more formally, the set has the property that any closed curve in the set is homotopic to a point\textemdash then we say that set is \emph{simply connected}.

\medskip

Next we will prove a lemma needed for our strengthened version of the Cauchy integral formula. A justification for its necessity is in order. The idea behind the proof of the upcoming ``new" Cauchy integral formula is simple. We start with the smooth closed contour we wish to integrate over, and shrink it to a new smooth closed contour that lies strictly inside the original contour. In this interior region, we can apply the regular Cauchy integral formula to closed smooth contours as the integrand is analytic by assumption. We then use a limit argument to push the inner contour via homotopy out to the boundary, allowing us to forgo the analyticity requirement of the function on the boundary while maintaining the original result. Unfortunately, the subtleties that arise when dealing with homotopic arguments are insurmountable by the assumed prerequisites for this article. It will be for this reason that we elect to apply a few famous high-powered theorems to prove the lemma.

\smallskip

In short, the following lemma guarantees us that the homotopy needed for the argument exists.

\begin{lemma}
\label{smooth homotopy}
Let $C$ be a smooth closed contour. Then there exists another smooth closed contour lying in $D^{+}$ that is homotopic to $C$. Moreover, the homotopy $h(s,t)$ of the contours has the following properties:

\begin{enumerate}[i]
\item\!\!. For fixed $s_0$, $h(s_0,t)$ describes a closed smooth contour as $t$ varies.
\item\!\!. $h(s,t)$ has non-trivial intersection with $C$ if and only if $s = 1$.
\end{enumerate}
\end{lemma}

\begin{proof}
The first result we need is the Riemann mapping theorem. It states that there exists a bijective holomorphic mapping from any strict open subset of the plane onto the open unit disk. It is also the case that the inverse function is holomorphic, a consequence of the inverse function theorem (see Appendix \ref{Appendix 2} for a precise statement).

Let $f: D^{+} \to \mathbb{D}$ be such a map from the interior of $C$ to the open unit disk. To make use of this special mapping in defining a homotopy, we need to extend the inverse mapping out to the boundary of $\mathbb{D}$, and then prove that the extension is of class $\mathcal{C}^1$. One of the results needed is known as Carathéodory's theorem. It states that $f$ can be extended to a homeomorphism between $\overline{D^{+}}$ and $\overline{\mathbb{D}}$ (see the book \cite{pommerenke2013boundary} by Pommerenke). This homeomorphism is what we will be referring to from here on when we write $f$. By definition, it then follows that $f^{-1} (\partial\mathbb{D}) = C$. 

\smallskip

Now that we have our extension, we need to show that it is continuously differentiable. Here we will use the result in Theorem 3.6 from \cite{pommerenke2013boundary} that says if we have a conformal mapping from the open unit disk to the interior of a $\mathcal{C}^1$ curve, then the first derivative of that mapping has a continuous extension to the boundary. In our case, this means that the derivative of $f^{-1}$ can be continuously extended to the boundary, which implies the extension of $f^{-1}$ is of class $\mathcal{C}^1$. If we define $h:[0,1] \times [0,2\pi] \to D^{+} \cup C$ by 
\smallskip
\[
h(s,t) = f^{-1}\left(\frac{1+s}{2}e^{it}\right)
\]

\noindent
then $h$ is a $\mathcal{C}^1$ homotopy from some smooth contour in $D^{+}$ to the contour $C$.

Since $f^{-1}$ is in particular injective, its derivative is nonzero at all points in $\mathbb{D}$, and consequently in $\overline{\mathbb{D}}$ by Theorem 6.4 in Lang's book \cite{lang2013complex}. Thus, we have shown that for each fixed $s_0$, the contour described by $h(s_0,t)$ is smooth. Finally, $f^{-1}$ being an injective map implies that the curve defined by $h(s,t)$ has non-trivial intersection with the contour $C$ only when $s = 1$.
\end{proof}

We now have the adequate tools to tackle the first part of the generalized Cauchy integral formula.

\begin{theorem}[Cauchy Integral Formula I]
\label{CIF}
Let $C$ be a smooth closed contour in the $\tau$-plane. Suppose that $f(z)$ is analytic in $D^{+}$, and continuous on $C$. Then
\smallskip
\[
\frac{1}{2\pi i} \int_{C} \frac{f(\tau)}{\tau - z}\,d\tau =
\begin{cases}
f(z) &\! \textup{if} \;\: z \in D^{+}\\
0    &\! \textup{if} \;\: z \in D^{-} .
\end{cases}
\]
\end{theorem}

\begin{proof}
For the first case, let $z$ be a point in $D^{+}$. With Lemma \ref{smooth homotopy} in mind, let $h(s,t)$ denote the homotopy between a smooth contour $C_0$ in $D^+$ that contains the point $z$ in its interior, and $C$ itself. Let $C_s$ denote the contour traced out by $h$ for fixed $s$. We can assume that $z$ lies in the interior of each $C_s$. Now, we can write the Cauchy integral as follows:
\smallskip
\[
\frac{1}{2\pi i}\int_{C_s} \frac{f(\tau)}{\tau-z} \,d\tau = \frac{1}{2\pi i}\int_{0}^{2\pi} \frac{f(h(s,t))}{h(s,t)-z} \cdot h_t(s,t) \,dt .
\]

Notice that the integrand on the left hand side is analytic on and inside the contour $C_s$. This allows us to apply the usual Cauchy integral formula to the expression. In addition, the integrand on the right hand side is uniformly continuous on the rectangle $[s,1]\times[0,2\pi]$. Taking the limit as $s \to 1^{-}$ and applying Lemma \ref{pass cont lim to integral},
\begin{align*}
f(z) = \frac{1}{2\pi i}\left[\lim_{s \to 1^{-}}\int_{C_s} \frac{f(\tau)}{\tau-z} \,d\tau\right] & = \frac{1}{2\pi i}\int_{0}^{2\pi} \lim_{s \to 1^{-}}\left[ \frac{f(h(s,t))}{h(s,t)-z} \cdot h_t(s,t)\right] \,dt \\
& = \frac{1}{2\pi i}\int_{0}^{2\pi}  \frac{f(h(1,t))}{h(1,t)-z} \cdot h_t(1,t) \,dt \\
& = \frac{1}{2\pi i}\int_{C} \frac{f(\tau)}{\tau-z} \,d\tau
\end{align*}

\noindent
which proves the result for points in the region $D^{+}$.

\smallskip

Now, if $z$ is in the region $D^{-}$, we can perform the same limit argument, although this time using the Cauchy-Goursat theorem on the intermediate contours. In this case, the integral evaluates to 0, and the same is true in the limit. This completes the proof.
\end{proof}

\smallskip

This next result is a version of the previous Cauchy integral formula for functions analytic on the \emph{exterior} of a closed contour.

\begin{theorem}[Cauchy Integral Formula II]
\label{CIF analytic in D-}
Let $C$ be a smooth closed contour in the $\tau$-plane. Suppose that the function $f(z)$ is analytic in $D^{-}$, has a holomorphic extension to the point at infinity, and is continuous on $C$. Then
\smallskip
\[
\frac{1}{2\pi i} \int_{C} \frac{f(\tau)}{\tau - z}\,d\tau =
\begin{cases}
f(\infty) &\! \textup{if} \;\: z \in D^{+}\\
-f(z) + f(\infty) &\! \textup{if} \;\: z \in D^{-} .
\end{cases}
\]

\end{theorem}

\begin{proof}
For any fixed $z$ in the plane and every point $\tau$ on $C$, there is $R > 0$ with the property that both $\labs z \rabs < R$ and $\labs \tau \rabs < R$. Let $C_R$ denote the circle of radius $R$ centred at the origin. Now, both the contour $C$ and the point $z$ lie entirely within the circle $C_R$. Write:
\smallskip
\[
\frac{1}{2\pi i} \int_{C} \frac{f(\tau)}{\tau - z}\,d\tau = \frac{1}{2\pi i} \int_{C - C_R} \frac{f(\tau)}{\tau - z}\,d\tau + \frac{1}{2\pi i} \int_{C_R} \frac{f(\tau)}{\tau - z}\,d\tau .
\]

First, we evaluate the integral on the right-hand side that is taken over $C-C_R$. We connect the contour $C$ and the circle $-C_R$ with a line that does not pass through the point $z$. This line is treated as an added path of integration that contributes nothing to the integral so long as we integrate along it in both possible directions of traversal. What this expression now represents is an integration along the circle in a clockwise direction, down the newly created line to the contour $C$, around $C$ in a counterclockwise direction, back up the line, and continuing along the circle until reaching the starting point. 

If $z\in D^{+}$, then $z$ is not enclosed by the path of integration.  Applying the Cauchy integral formula as stated in Theorem $\ref{CIF}$, the integral evaluates to 0. If $z\in D^{-}$, then by design our path construction encloses the point. Reversing the direction of integration and applying Theorem $\ref{CIF}$ gives us a value of $-f(z)$. Now for the integral over the circle $C_R$. Let $\zeta = 1/\tau$. Then $d\zeta = -\zeta^2d\tau$, and we can write
\smallskip
\[
\frac{1}{2\pi i} \int_{C_R} \frac{f(\tau)}{\tau - z}\,d\tau  = -\frac{1}{2\pi i} \int_{C_{\frac{1}{R}}} \frac{f\left(\frac{1}{\zeta}\right)}{\left(\frac{1}{\zeta} - z\right) \zeta^2}\,d\zeta .
\]

Define $g(\zeta) \coloneqq f(1/\zeta)$. As $f$ is analytic on the unbounded complement of the circle $C_R$, $g$ is analytic on the bounded complement of the circle $C_{1/R}$, including the point 0 by our assumption that $f$ can be holomorphically extended to the point at infinity. In fact, $g$ is analytic on the contour $C_{1/R}$ as well.

When 0 and $\infty$ are in different connected components with respect to a contour in the plane, a transformation of the form $\tau \mapsto 1/\tau$ reverses the direction the parameterization of said contour. In our case, this means that the circle $C_{1/R}$ is traversed clockwise, and hence $-C_{1/R}$ is traversed counterclockwise. The integral can now be written as:
\smallskip
\[
\frac{1}{2\pi i} \int_{C_R} \frac{f(\tau)}{\tau - z}\,d\tau  = \frac{1}{2\pi i} \int_{-C_{\frac{1}{R}}} \frac{g(\zeta)}{\left(1 - \zeta z\right) \zeta}\,d\zeta . \tag{1}
\]

The term $1-\zeta z$ in the denominator of the integrand is nonzero in the interior of $-C_{1/R}$, due to the fact that $1/\labs z\rabs > 1/R$. Thus, the function $g(\zeta)/(1-\zeta z)$ is analytic on the interior of the contour of integration. However, the interior of the circle contains the origin, so $\zeta = 0$ is a singularity of the integrand in the interior of $-C_{1/R}$. Applying Theorem \ref{CIF} to (1):
\smallskip
\[
\frac{1}{2\pi i} \int_{C_R} \frac{f(\tau)}{\tau - z}\,d\tau = \frac{g(0)}{1-0z} = f(\infty) .
\]

Notice that for this second integral, the calculation is independent of the component in which $z$ resides. Combining the results completes the proof.
\end{proof}

The Cauchy integral formula links the boundary behaviour of a function to the two domains lying on either side of the contour $C$. What we want to do next is shift our perspective of the Cauchy integral formula to not merely being a convenient method solving integrals, but as function itself with a unique reproducing property. We sum this up with the following definition of a special singular integral operator called the \emph{Cauchy type integral}.

\begin{definition}[Cauchy Type Integral]
Let $C$ be a smooth contour in the $\tau$-plane. Suppose that $\varphi: C \to \mathbb{C}$ is a continuous function on $C$. Then the function
\smallskip
\[
\Phi(z) = \frac{1}{2\pi i} \int_{C} \frac{\varphi (\tau)}{\tau -z}\,d\tau
\]

\noindent
is called the Cauchy type integral. $\varphi (\tau)$ is referred to as its density, and $1/(\tau-z)$ the (Cauchy) kernel. 
\end{definition}

Looking ahead, the Cauchy type integral will be the main object we analyze. More specifically, we are going to apply analytic techniques in an attempt to understand its behaviour near the contour of integration. To do so, we must gain a firm understanding of it in a more tractable region on the plane.

This next result shows that the function $\Phi(z)$ is analytic everywhere in the plane except for points along the contour $C$. We refer to this line as the \emph{singular line} of the Cauchy type integral.

\begin{theorem}
\label{CTI is analytic}
Let $C$ be a smooth contour in the $\tau$-plane and let $D \subset \mathbb{C}$ be a domain in the z-plane. Let $f: C \times D \rightarrow \mathbb{C}$ be a complex-valued function of two complex variables. Assume that the following conditions hold:
\begin{enumerate}[i]
    \item\!\!. $f(\tau, z)$ is continuous
    \item\!\!. $f(\tau, z)$ analytic with respect to $z$
    \item\!\!. $f_{z} (\tau, z)$ is continuous in $\tau$ when the $z$-variable is fixed.
\end{enumerate}

Then the integral-defined function
\smallskip
\[
F(z) = \int_{C} f(\tau, z)\,d\tau
\]

\noindent 
is analytic in $D$, and the expression for $F'$ is obtained by differentiating under the integral sign. 
\end{theorem}

\begin{proof}
Fix $z_0 \in D$, and let $K \subset D$ be a closed, bounded disk centred at $z_0$. Define a new function $g: C \times K \to \mathbb{C}$ by
\smallskip
\[
g(\tau,z) =
\begin{cases}
\frac{f(\tau,z) - f(\tau,z_0)}{z-z_0} &\! \textup{if} \;\: z \neq z_0\\
f_z(\tau,z_0)    &\! \textup{if} \;\: z = z_0 .
\end{cases}
\]

\noindent
First, it is obvious that $g$ is continuous at points of the form $(\tau,z)$ when $z \neq z_0$. We need to be slightly more careful with checking continuity at points of the form $(\tau, z_0)$. 

Let $(\tau',z_0) \in C \times K$, and let $\varepsilon > 0$ be given. By the above observation and analyticity of $f$ in $z$, we can arrange for $\delta_1 > 0$ such that
\smallskip
\[
\labs \frac{f(\tau,z)-f(\tau,z_0)}{z-z_0} - \frac{f(\tau',z)-f(\tau',z_0)}{z-z_0} \rabs < \frac{\varepsilon}{2} \;\;\;\text{and}\;\;\; \labs \frac{f(\tau',z)-f(\tau',z_0)}{z-z_0} - f_z(\tau',z_0) \rabs < \frac{\varepsilon}{2}
\]

\noindent
whenever $0< \labs z - z_0 \rabs < \delta_1$ and $\labs \tau - \tau' \rabs < \delta_1$. It follows that
\smallskip
\[
\labs \frac{f(\tau,z)-f(\tau,z_0)}{z-z_0} - f_z(\tau',z_0) \rabs < \varepsilon
\]

\noindent
when $z$ and $\tau$ satisfies these conditions. Now by continuity of $f_z$ in $\tau$, there is $\delta_2 > 0$ such that 
\smallskip
\[
\labs f_z(\tau, z_0) - f_z(\tau', z_0) \rabs < \varepsilon
\]

\noindent
whenever $\labs \tau - \tau' \rabs < \delta_2$. Taking $\delta = \min\{\delta_1, \delta_2\}$, it is clear that the bound
\smallskip
\[
\labs g(\tau, z) - g(\tau', z_0) \rabs < \varepsilon
\]

\noindent
holds so long as $\| (\tau,z) - (\tau',z_0) \| < \delta$. Hence $g$ is continuous, and moreover uniformly continuous since $C\times K$ is compact. Thus Corollary \ref{pass cont lim to contour integral} can be applied to obtain
\smallskip
\begin{align*}
\lim_{z \to z_0}\left[ \frac{F(z)-F(z_0)}{z-z_0} - \int_{C} f_z(\tau,z_0) \,d\tau \right] &= \lim_{z \to z_0}\Bigg[  \int_{C} \left[ \frac{f(\tau,z)-f(\tau,z_0)}{z-z_0}-f_z(\tau,z_0) \right] \,d\tau \Bigg] \\ &= \lim_{z \to z_0}\Bigg[ \int_{C} \left[ \,g(\tau,z) - g(\tau,z_0) \right] \,d\tau \Bigg] \\
&= 0
\end{align*}

\noindent
proving analyticity of $F$ in $D$.
\end{proof}

\begin{corollary}
The Cauchy type integral function is analytic everywhere except for points that coincide with the contour of integration.
\end{corollary}

\begin{proof}
The details are left to the reader.
\end{proof}

For the case that the contour $C$ is closed, the function $\Phi$ representing the Cauchy type integral can more accurately be thought of as splitting into two different functions. If $z \in D^{+}$ we use the notation $\Phi^{+}$, and if $z \in D^{-}$ we use $\Phi^{-}$. The reason for this distinction is that in general we cannot analytically continue one into the other across $C$.

\smallskip

At the moment, the behaviour of the Cauchy type integral is mysterious at points along the contour. This problem is analogous to the determining the convergence and divergence of improper integrals encountered in elementary calculus. Our main goal in the upcoming sections will be to clarify this behaviour.

\smallskip

Before moving on, let us note the following important property:

\begin{proposition}
\label{CTI vanish at infinity}
The Cauchy type integral function $\Phi$ can be holomorphically extended to the point at infinity, where it takes on a value of 0.
\end{proposition}

\begin{proof}
We are going to build a series representation of the Cauchy type integral in the neighbourhood of infinity to derive the result. Fix $z \in D^{-}$ far enough from the origin so that it satisfies $\labs z \rabs > \labs \tau\rabs$ for every point $\tau$ on the contour $C$. Expanding the kernel into a geometric series:
\begin{align*}
\frac{1}{\tau-z} &= -\frac{1}{z} \cdot \frac{1}{1-\frac{\tau}{z}} \\
& = -\frac{1}{z} \sum_{n=0}^{\infty} \left( \frac{\tau}{z} \right)^n \\
& = -\sum_{n=1}^{\infty} \frac{\tau^{n-1}}{z^n}
\end{align*}

\noindent
This series will converge for $\labs \tau \rabs < \labs z \rabs$. Now we multiply through by the density over a factor of $2\pi i$, and then integrate over $C$ with respect to $\tau$ to obtain the expression
\smallskip
\[
\frac{1}{2\pi i} \int_{C} \frac{\varphi(\tau)}{\tau-z} \,d\tau = -\frac{1}{2\pi i} \int_{C} \left[ \varphi(\tau) \sum_{n=1}^{\infty} \frac{\tau^{n-1}}{z^n} \right]\,d\tau .
\]

On the left hand side, we have the Cauchy type integral $\Phi(z)$. For the right hand side, note that we can pull the sum out of the integrand whenever integrating the product of a convergent power series with a continuous function (see Section 65 of \cite{brown2009complex}). This technique allows one to integrate a power series term by term. In view of our choice of $z$ and continuity of the density along the contour $C$, we arrive at
\smallskip
\[
\Phi(z) = \sum_{n=1}^{\infty} \frac{a_n}{z^n}
\]

\noindent
where 
\smallskip
\[
a_n = -\frac{1}{2\pi i} \int_{C} \tau^{n-1}\varphi(\tau) \,d\tau .
\] 

\noindent
Since $\Phi\left(1/z\right) = \sum_{n=1}^\infty a_n z^n$ has a removable singularity at $z=0$, $\Phi$ can be holomorphically extended to the point at infinity, and $\Phi(\infty)$ has the announced value.
\end{proof}

\section{The Hölder Condition}

To get a handle on the Cauchy type integral along the singular line, we need a particularly nice density function. What is meant by ``nice" in this context will be the focus of this section. Just as uniform continuity is a strong form of continuity in the sense that it requires that a $\delta$ can be found that works over the entire domain, uniform continuity can be strengthened by restricting how \emph{quickly} function values can approach each other.

\begin{definition}[The Hölder Condition]
\label{Holder condition}
Let $D$ be a subset of the complex plane, and let $\varphi: D\rightarrow \mathbb{C}$ be a complex-valued function defined on $D$. We say that $\varphi$ satisfies the Hölder condition on $D$ if there exists constants $A > 0$ and $0 <\lambda \leq 1$ having the property that
\smallskip
\[
\labs\varphi(z) - \varphi(w)\rabs\leq A\labs z - w\rabs^{\lambda}
\]

\noindent 
whenever $z$, $w \in D$. Alternatively, we say that $\varphi$ is $\lambda$-Hölder, or simply Hölder continuous. We refer to $A$ as the Hölder constant, and $\lambda$ as the Hölder index.
\end{definition}

\begin{remark}
Values of $\lambda$ exceeding 1 are not under consideration due to redundancy. For if $\varphi$ satisfies the above condition for some $\lambda > 1$, then rearranging the inequality we find that
\smallskip
\[
0 \leq \labs \frac{\varphi(z) - \varphi(w)}{z - w} \rabs \leq A\labs z-w\rabs^{\lambda-1} .
\]

As $z\rightarrow w$, the difference quotient goes to 0 by the squeeze theorem. This implies that $\varphi$ is differentiable on $D$, where $\varphi' \equiv 0$, forcing $\varphi$ to be constant on each connected component of $D$. 
\end{remark}

As the naming convention may suggest, the Hölder index is the essential characteristic of the condition, and not so much the constant. Its size is what classifies the Hölder functions, and ultimately controls the behaviour of the function.

Note that the larger the value of $\lambda$, the more severe the restriction. Thus the case $\lambda = 1$ encapsulates the smallest class of these functions, which are often referred to as being \emph{Lipschitz continuous}, or satisfying the Lipschitz condition.

\smallskip

Next, we establish some basic properties of Hölder continuous functions.

\begin{proposition}
\label{holder implies uniform}
A Hölder continuous function $\varphi$ is in particular uniformly continuous.
\end{proposition}

\begin{proof}
Let  $\lambda$ and $A$ denote the Hölder index and constant of $\varphi$ in a set $D$, respectively, and let $\varepsilon > 0$ be given. Choose $\delta = \left(\varepsilon/A\right)^{1/\lambda}$. If $z$, $w \in D$ satisfy $\labs z-w\rabs < \delta$, by Hölder continuity of $\varphi$ in $D$ we have
\smallskip
\[
\labs \varphi(z) - \varphi(w) \rabs \leq A\labs z-w\rabs^{\lambda} < \varepsilon .
\]
\end{proof}

\begin{example}
The prototypical example of a uniformly continuous function that does not satisfy the Hölder condition is the following. Let $f:\left[0,\frac{1}{2}\right] \to \mathbb{R}$ be defined as
\smallskip
\[
f(x) =
\begin{cases}
\frac{1}{\ln(x)} &\text{if} \;\; 0 < x \leq \frac{1}{2} \\
0 &\text{if} \;\; x=0 .
\end{cases}
\]

It is readily observed that $f$ constitutes a uniformly continuous function. However, for any combination of constant $A > 0$ and index $0 < \lambda \leq 1$, we can always violate the bound in Definition \ref{Holder condition}. Recall from elementary calculus that for every $\lambda > 0$,
\smallskip
\[
\lim_{x \to 0^{+}} x^{\lambda}\ln(x) = 0.
\]

\noindent
In words, the positive power functions completely dominate the logarithm near zero. So for any constant $A>0$, we can certainly find $x$ sufficiently close to 0 so that
\smallskip
\[
\labs x^{\lambda} \ln(x) \rabs < \frac{1}{A}
\]

\noindent
whence
\smallskip
\[
\labs f(x) - f(0) \rabs = \labs \frac{1}{\ln(x)} \rabs > A\labs x \rabs ^{\lambda} .
\]
\end{example}

\begin{proposition}
Suppose that $\alpha$ and $\beta$ satisfy the Hölder condition on a smooth contour $C$, with Hölder indices $\lambda$ and $\mu$, respectively. Then both their sum and product satisfy the Hölder condition on $C$. If additionally, $\beta$ does not vanish on $C$, then their quotient satisfies the Hölder condition as well. In each case, the new Hölder index will be the smaller of the two original indices.
\end{proposition}

\begin{proof}
We shall prove each result in succession to make use of previously established bounds without repeated justification.

Let A and B denote the Hölder constants of $\alpha$ and $\beta$, respectively. As we are assuming the length of $C$ to be a finite quantity, by definition the distance between any two points lying on $C$ is no greater than the length of the contour $C$, denoted by $\ell (C)$. Without loss of generality, assume $\lambda \leq \mu$. Set $\varepsilon \coloneqq \mu - \lambda$, $M \coloneqq \text{max}\{A,B\}$, and let $z$, $w \in C$. 

\smallskip

First, we get an estimate for their sum:
\smallskip
\begin{align*}
\labs \left( \alpha(z) + \beta(z) \right) - \left( \alpha(w) + \beta(w) \right) \rabs & \leq \labs \alpha(z) - \alpha(w)\rabs + \labs \beta(z) - \beta(w) \rabs\\
& \leq A\labs z-w \rabs ^{\lambda} + B\labs z-w \rabs ^{\mu} \\
& \leq M \labs z-w \rabs ^{\lambda} \left( 1+\labs z-w \rabs^{\varepsilon} \right) \\
& \leq M \left( 1+\ell (C)^{\varepsilon} \right) \labs z-w \rabs ^{\lambda}
\end{align*}

\noindent
which is what we wanted to show. Now, $C$ is compact, so we can find $R>0$ with the property that 
\smallskip
\[
\labs \alpha (z) \rabs \leq R \quad\text{and}\quad \labs \beta (z) \rabs \leq R 
\]

\noindent
for every $z \in C$. We can now obtain an estimate for the product:
\smallskip
\begin{align*}
\labs \alpha(z)\beta(z) - \alpha(w)\beta(w) \rabs & \leq \labs \beta(z) \rabs \labs \alpha(z) - \alpha(w) \rabs + \labs \alpha(w) \rabs \labs \beta(z)-\beta(w) \rabs  \\
& \leq R \left( \labs \alpha(z) - \alpha(w)\rabs + \labs \beta(z) - \beta(w) \rabs \right)\\
& \leq RM(1+\ell (C)^{\varepsilon})\labs z-w \rabs^{\lambda} .
\end{align*}

Finally, we examine the case for quotients. Assume that $\beta(z) \neq 0$ for every point $z$ on the contour $C$. By continuity, we can bound $\labs\beta(z)\rabs$ away from 0 by a positive constant on a neighbourhood of every point on the contour. An application of the Heine-Borel theorem allows us find a uniform lower bound on $\beta$. That is, there is $K > 0$ with the property that $K \leq \labs \beta \rabs$ on $C$. Then we obtain
\smallskip
\begin{align*}
\labs \frac{\alpha(z)}{\beta(z)} - \frac{\alpha(w)}{\beta(w)}\rabs & \leq \frac{\labs \beta(w) \rabs \labs \alpha(z)-\alpha(w) \rabs + \labs \alpha(w) \rabs \labs \beta(z)-\beta(w) \rabs}{\labs \beta(z)\beta(w)\rabs} \\
& \leq \frac{RM(1+\ell(C)^{\varepsilon})}{K^2} \labs z-w \rabs^{\lambda}
\end{align*}

\noindent
as desired.
\end{proof}

\begin{proposition}
\label{smaller holder}
Suppose that $\varphi$ is $\lambda$-Hölder on a smooth contour $C$. Then for any $0<\alpha<\lambda$, $\varphi$ is $\alpha$-Hölder on $C$.
\end{proposition}

\begin{proof}
As $C$ is rectifiable, it has finite length $\ell (C)$. Let $A > 0$ denote the Hölder constant of $\varphi$ on $C$. For any pair of points $z$, $w\in C$, we have
\smallskip
\begin{align*}
\labs \varphi (z) - \varphi (w) \rabs  & \leq A \labs z - w \rabs ^{\lambda} \\
& = A \labs z - w\rabs^{\lambda - \alpha} \labs z - w \rabs^{\alpha} \\
& \leq A \ell (C)^{\lambda - \alpha} \labs z - w \rabs^{\alpha} .
\end{align*}
\end{proof}

\begin{remark}
The previous two propositions above can be shown to remain true in general for arbitrary subsets of the complex plane, so longs as the sets have finite diameter.
\end{remark}

\begin{example}
As was shown by Proposition $\ref{smaller holder}$, the classes of Hölder continuous functions form a sort of containment with each other in the sense that functions which satisfy the condition for larger indices also satisfy it for smaller indices. It turns out that this containment is strict. For example, take the function $f(x) = \sqrt{x}$ defined on the interval $[0,1]$. It is straightforward to show that $f$ is Hölder continuous with index $\frac{1}{2}$. Let $a$, $b \in [0,1]$. Then
\smallskip
\[
\labs \sqrt{a}-\sqrt{b}\rabs^2 \leq  \labs \sqrt{a}+\sqrt{b} \rabs \labs \sqrt{a}-\sqrt{b} \rabs  = \labs a -b\rabs .
\]

\noindent
Taking the square root of both sides lends the desired result. This happens to be the \emph{maximum} Hölder index we can achieve on this interval. For suppose that $0<\varepsilon\leq\frac{1}{2}$ and $A > 0$. Then 
\smallskip
\begin{align*}
1 &> \frac{A}{A+1} \\
\Rightarrow \labs \frac{1}{\left( A+1 \right)^{\frac{1}{\varepsilon}}} \rabs^{\frac{1}{2}} &> A \labs \frac{1}{\left( A+1 \right)^{\frac{1}{\varepsilon}}}\rabs ^{\frac{1}{2} + \varepsilon} \\
\Rightarrow \labs \frac{1}{\left( A+1 \right)^{\frac{1}{2\varepsilon}}} \rabs &> A \left( \frac{3}{4} \right)^{\frac{1}{2} + \varepsilon} \labs \frac{1}{\left( A+1 \right)^{\frac{1}{\varepsilon}}}\rabs ^{\frac{1}{2} + \varepsilon} \\
\Rightarrow \labs \frac{1}{\left( A+1 \right)^{\frac{1}{2\varepsilon}}} -\frac{1}{2\left( A+1 \right)^{\frac{1}{2\varepsilon}}} \rabs &> A \labs \frac{1}{\left( A+1 \right)^{\frac{1}{\varepsilon}}} - \frac{1}{4\left( A+1 \right)^{\frac{1}{\varepsilon}}}\rabs ^{\frac{1}{2} + \varepsilon} .
\end{align*}

\noindent
If we let $a=\left( A+1 \right)^{-1/\varepsilon}$ and $b= \left[ 4\left( A+1 \right)^{1/\varepsilon} \right]^{-1}$, then $a$, $b \in [0,1]$, and by the final inequality above:
\smallskip
\[
\labs \sqrt{a} - \sqrt{b} \rabs > A \labs a - b\rabs^{\frac{1}{2} + \varepsilon}
\]

\noindent
so the Hölder condition is not satisfied, thus proving our claim. We avoided choosing one of the points to be zero to emphasize that the problem point is \emph{not} zero itself. The limiting factor on the Hölder index is in the \emph{rate of change} of the function \emph{near} 0.
\end{example}

\begin{proposition}
\label{holder_composition}
Let $f$ be $\alpha$-Hölder on $D$, and suppose that $g$ is $\lambda$-Hölder on the image of $f$. Then the composition $g \circ f$ is $\lambda\alpha$-Hölder on D.
\end{proposition}
\smallskip
\begin{proof}
Let $A$, $B > 0$ denote the Hölder constants associated with $g$ and $f$, respectively. Let $z_1, z_2$ be points in the set D. We have
\begin{align*}
\labs (g \circ f) (z_1)- (g \circ f) (z_2)\rabs & \leq A \labs f(z_1) - f(z_2)\rabs^{\lambda} \\
& \leq AB^{\lambda} \labs z_1 - z_2 \rabs^{\lambda\alpha} .
\end{align*}
\end{proof}

Often, both the Lipschitz and general Hölder condition are thought of as \emph{local} properties. Given some well-behaved function, different subsets of its domain may require different indices or constants, let alone continue to satisfy the condition. This notion is made precise with the following definition.

\begin{definition}[Local Lipschitz Condition]
We say that a function $f: D \to \mathbb{C}$ is locally Lipschitz at a point $z \in D$ if there is a neighbourhood $B$ of $z$ such that $f$ satisfies the Lipschitz condition in $D\cap B$. That is, there is $r > 0$ and a constant $M > 0$ (possibly depending on $z$) such that 
\smallskip
\[
\labs f(z_1) - f(z_2) \rabs \leq M \labs z_1 - z_2\rabs
\]

\noindent
for all $z_1$, $z_2 \in B(z,r) \cap D$. If $f$ is locally Lipschitz at every point in $D$, then we say $f$ is locally Lipschitz on $D$.
\end{definition}

\smallskip

\begin{remark}
The above definition generalizes to a local Hölder condition in the obvious way. We will not explicitly need it however.
\end{remark}

In certain contexts, the distinction between locally Lipschitz and simply Lipschitz is be made by referring to the latter as \emph{globally} Lipschitz, calling attention to the fact that we can find a Lipschitz constant that works for \emph{any} pair of points in the set.

\smallskip

The rest of this section will be dedicated to identifying familiar function and domain conditions that imply Lipschitz continuity.

\begin{proposition}
\label{analytic_implies_local_lipschitz}
If $f$ is analytic on a subset $D$ of the complex plane, then $f$ is locally Lipschitz on D.
\end{proposition}

\begin{proof}
There must be an open set $\Omega$ containing $D$ in which $f$ is analytic, so fix $z_0 \in D$ and take $r > 0$ such that $\overline{B(z_0,r)} \subset \Omega$. Using analyticity of $f$ and a consequence of the maximum modulus principle, we can find $w_0 > 0$ on the boundary of $B(z_0,r)$ with the property that $\labs f'(w) \rabs \leq \labs f'(w_0) \rabs$ for every $w \in B(z_0,r)$. Choose a point $z$ that lies in both $D$ and the open ball $B(z_0,r)$. Since open balls are convex, we can connect $z_0$ and $z$ with a straight line $L$ that is entirely contained within $B(z_0,r)$. Combining the fundamental theorem of contour integrals and the estimation lemma, we find
\smallskip
\[
\labs f(z) - f(z_0) \rabs = \labs \int_{L} f'(z) \,dz \rabs \leq \labs f'(w_0) \rabs \labs z-z_0 \rabs .
\]
\end{proof}

We should expect that when formulating a local version of a property, the global condition should imply the local condition immediately. What is often of interest is what is needed to make the converse direction true. That is, if we are given that the local condition is satisfied everywhere, what extra conditions are needed to guarantee we have the global condition? This theorem answers that question for the Lipschitz condition.

\begin{theorem}
\label{local_lipschitz_implies_global}
If a function $f$ is locally Lipschitz on a compact subset $K$ of the complex plane, then it is globally Lipschitz on $K$.
\end{theorem}

\begin{proof}
For the sake of contradiction, suppose that $f$ is not globally Lipschitz. Then the quotient 
\smallskip
\[
\labs \frac{f(z) - f(w)}{z - w} \rabs \tag{1}
\]

\noindent
is unbounded on the set $K \times K - \text{diag}(K)$, where $\text{diag}(K) = \{(\tau,\tau): \tau \in K \}$. This means that for each $n \in \mathbb{N}$, we can find $(z_n, w_n) \in K \times K - \text{diag}(K)$ such that
\smallskip
\[
\labs \frac{f(z_n) - f(w_n)}{z_n - w_n} \rabs \geq n . \tag{2}
\]

Since $f$ is locally Lipschitz in $K$, $f$ is in particular continuous on $K$, and hence bounded on $K$. It follows that the numerator in (1) is bounded on $K$. Thus for the inequality in (2) to be true on $K \times K - \text{diag}(K)$, it must be the case that
\smallskip
\[
\labs z_n - w_n \rabs \to 0 \quad\text{as}\quad n \to \infty.
\]

By the Bolzano-Weierstrass theorem, we can find a subsequence $(z_{n_k})_{k=1}^{\infty}$ of $(z_n)_{n=1}^{\infty}$ that converges to a complex number $z_0$. Moreover, as $K$ is closed, this limit point must be in the set $K$. We claim that the subsequence $(w_{n_k})_{k=1}^{\infty}$ also converges to $z_0$. Consider the inequality
\smallskip
\[
\labs w_{n_k} - z_0 \rabs \leq \labs w_{n_k} - z_{n_k} \rabs + \labs z_{n_k} - z_0 \rabs.
\]

For $N$ sufficiently large and $n_k \geq N$, both terms on the right hand side can be made arbitrarily small. 

What we have established contradicts the assumption that $f$ is locally Lipschitz at $z_0$. To illustrate why, recall that by assumption there is an $r > 0$ such that $f$ is Lipschitz in $B(z_0,r)\cap K$ with constant $M_{z_0} > 0$. We just showed that there is $N \in \mathbb{N}$ such that $z_{n_k}$, $w_{n_k} \in B(z_0, r) \cap K$ for $k \geq N$. By the Archimedean property and the convergence of the subsequences, we can find an index $n_{k_0} > M_{z_0}$ so large as to satisfy $k_0 \geq N$. By (2), this gives us the inequality 
\smallskip
\[
\labs \frac{f(z_{n_{k_0}}) - f(w_{n_{k_0}})}{z_{n_{k_0}} - w_{n_{k_0}}} \rabs \geq{n_{k_0}}
\]

\noindent
and in particular
\smallskip
\[
\labs f(z_{n_{k_0}}) - f(w_{n_{k_0}}) \rabs > M_{z_0}\labs z_{n_{k_0}} - w_{n_{k_0}} \rabs .
\]

Therefore, $f$ must be globally Lipschitz.
\end{proof}

\smallskip

\begin{remark}
Compactness is essential to this theorem. To illustrate why, consider the two analytic functions $z \mapsto 1/z$ and $z \mapsto z^2$. Both are first examples of functions that are not uniformly continuous on domains which contain a deleted neighbourhood of 0 and $\infty$, respectively. By Proposition $\ref{holder implies uniform}$, neither function is Hölder continuous on these domains, but they are locally Lipschitz at each point by Theorem $\ref{analytic_implies_local_lipschitz}$. Indeed, neither domain is compact. 
\end{remark}

\smallskip

\begin{corollary}
\label{Analytic fn are lipschitz on smooth contours}
Analytic functions satisfy the Lipschitz condition on smooth contours.
\end{corollary}

\smallskip

\begin{corollary}
Let $f$ be analytic on a compact set $K \subset \mathbb{C}$. If $g$ is $\lambda$-Hölder on the image of $f$, then the composition $g \circ f$ is $\lambda$-Hölder on $K$.
\end{corollary}

\begin{proof}
The result follows at once from Theorem \ref{analytic_implies_local_lipschitz}, and Proposition \ref{holder_composition} and \ref{local_lipschitz_implies_global}.
\end{proof}

As parameterizations of contours are defined on compact intervals of the real line, but are complex-valued functions, it will be useful to have some results targeted to this particular class of functions.

\begin{lemma}
\label{realLip}
Let $f:[a,b] \to \mathbb{R}$ be a differentiable function with bounded derivative on $(a,b)$. Then $f$ is Lipschitz continuous.
\end{lemma}

\begin{proof}
By assumption, we can uniformly bound $f'$ by some constant $M > 0$ on $(a,b)$. For any pair of points $x_1$, $x_2 \in [a,b]$ with $x_2 > x_1$, we apply the Mean Value Theorem to obtain
\smallskip
\[
\frac{f(x_2) - f(x_1)}{x_2 - x_1} = f'(c)
\]

\noindent
for some $c \in (x_1,x_2)$. Therefore
\smallskip
\[
\labs f(x_2) - f(x_1) \rabs = \labs f'(c) \rabs \labs x_2 - x_1 \rabs \leq M \labs x_2 - x_1 \rabs
\]

\noindent
and the claim follows.
\end{proof}

\begin{example}
The converse to Lemma \ref{realLip} is not true in general. That is, satisfying the Lipschitz condition on a subset of $\mathbb{R}$ is not strong enough to imply differentiability everywhere on that set. A simple example is the real absolute value function, which is obviously Lipschitz continuous everywhere, but not differentiable at the origin. However, real Lipschitz continuous functions can be shown to be absolutely continuous, which is enough to guarantee differentiability \emph{almost} everywhere.
\end{example}

\begin{theorem}
\label{paramLip}
Let $\alpha: [a,b] \to \mathbb{C}$ be a complex-valued function of one real variable. Suppose that $\alpha$ is differentiable and has bounded derivative on $(a,b)$. Then $\alpha$ is Lipschitz continuous.
\end{theorem}

\begin{proof}
Let $x(t)$ and $y(t)$ denote the real and imaginary parts of the function $\alpha$, so that $\alpha(t) = x(t) + iy(t)$. Then the derivative of the function is given by $\alpha'(t) = x'(t) + iy'(t)$. By assumption, there is a uniformly bounding constant $M > 0$ for the function $\alpha'$. Note that the bound
\[
\labs x' \rabs\text{, }\labs y' \rabs \leq \labs \alpha' \rabs \leq M \tag{1}
\]

\noindent
holds on $(a,b)$. Since both of $x(t)$ and $y(t)$ are real differentiable functions, by Lemma $\ref{realLip}$ they are Lipschitz continuous, and $M$ can be used as the Lipschitz constant by (1). Letting $t_1$, $ t_2 \in [a,b]$, we find
\smallskip
\begin{align*}
\labs \alpha(t_1) - \alpha(t_2) \rabs^2 & = \left( x(t_1) - x(t_2) \right)^2 +\left( y(t_1) - y(t_2) \right)^2 \\
& = \labs x(t_1) - x(t_2) \rabs^2 +\labs y(t_1) - y(t_2) \rabs^2\\
& \leq 2M^2 \labs t_1 - t_2\rabs^2 .
\end{align*}

\noindent
thus the conclusion follows.
\end{proof}

\section{Principal Value of a Singular Integral}
\label{Principal Value Integral}

In this section, we look at a concept known as a principal value: a way of assigning a meaningful value to improper integrals where the integrand contains a singularity. This is with the view of determining when it is possible to extract a usable value for the Cauchy type integral along the singular line. We begin with a definition.

\begin{definition}[Principal Value Integral]
Let $C$ be a smooth contour in the $\tau$-plane, and let $\tau_0$ be a point on $C$. Let $C_r \coloneqq C - B(\tau_0,r)$ denote the contour $C$ with a neighbourhood of radius $r>0$ around the point $\tau_0$ removed. Let $\varphi$ be an integrable function on $C_r$. For the singular integral with Cauchy kernel and $\varphi$ density
\smallskip
\[
\int_C \frac{\varphi (\tau)}{\tau -\tau_0}\,d\tau
\]

\noindent
we define its principal value as
\smallskip
\[
P.V.\int_{C} \frac{\varphi (\tau)}{\tau -\tau_0}\,d\tau \coloneqq \lim_{r \to 0^{+}} \int_{C_r}\frac{\varphi (\tau)}{\tau -\tau_0}\,d\tau .
\]

\end{definition}

\begin{remark}
It can be shown that for any point $\tau_0$ on a smooth contour, there exists an open ball centered at $\tau_0$ that contains a single arc of the contour $C$. This a subtle, but important fact because we want to maintain as much of the curve as possible. When examining the limit, it would be awfully inconvenient to have to consider the possibility of portions of the curve being deleted that are close to the the singularity in space, but far enough away in terms of distance along the curve as to not affect what the limit is attempting to capture. The proof requires a topological argument, which is given in Appendix \ref{Appendix 1}.
\end{remark}

Typically, we only talk about the principal value of an integral when a singularity is present in the integrand. However, if we were to use the principal value definition on a function without a singularity, it would produce the same value as the regular contour integral definition. To see why, let $f$ be continuous on a smooth contour $C$, $\tau_0$ a point on $C$, and $r>0$ small enough so that $\mathcal{B}_r = B(\tau_{0},r)\cap C$ contains one connected component. By the estimation lemma:
\smallskip
\[
\labs \int_{\mathcal{B}_r} f(\tau) \,d\tau \rabs \leq \sup_{\tau\in C} \labs f(\tau) \rabs \cdot \ell(B_r) .
\]

The length of $\mathcal{B}_r$ approaches 0 as $r$ goes to 0, hence so does the above integral. It then follows that the contour integral and principal value integral of $f$ over $C$ will be the same.

\medskip

For an intuition-building example, consider the real-valued function $f(x) = 1/x$ on $[-1,0)\cup(0,1]$. Interpreting the improper integral of $f$ on this interval as the area under the curve, we could convince ourselves that\textemdash in some sense\textemdash the ``areas" bounded by $f$ above and below the $x$-axis cancel each other out, so the integral should be 0. However, as one learns in a course of integral calculus, this is a divergent integral under the standard notion of improper integrals. If instead we use a principal value integral, not only does the integral exists, but the result is indeed zero. It is left to the inquisitive reader to investigate how the values of this integral with respect to a principal value changes if the definition is modified to allow differing approaches on either side of the singularity.

\medskip
\begin{remark}
While symmetry is an essential property in the deletion process of the principal value integral to ensure that one direction of approach is not favoured over the other, it is not strictly necessary to delete a \emph{neighbourhood} of $\tau_0$ in the limiting process. As we will see in the following results, the important characteristic is that the \emph{ratio} between the distance from $\tau_0$ to each of the new end points created through deletion approaches 1. That is, the symmetry must exist in the limit, but not necessarily anywhere else. Using basic open balls, this ratio is always 1. In more advanced or delicate treatments, the redefinition is necessary.
\end{remark}

\medskip

The following lemma gives us a formula for the principle value of the simplest singular integral.

\smallskip
\begin{lemma}
\label{ipi_integral}
Let $C$ be a smooth contour and $\tau_0$ be a point on $C$ that does not coincide with its endpoints. Then the principal value singular integral with Cauchy kernel and unit density
\smallskip
\[
\int_{C} \frac{d\tau}{\tau-\tau_0}
\]

\noindent
exists. If $C$ is open with true endpoints $a$ and $b$, the formula for its principal value is given by
\smallskip
\[
P.V.\int_{C} \frac{d\tau}{\tau-\tau_0} = \log(b-\tau_0)-\log(a-\tau_0)+\pi i.
\]

If $C$ is closed, the value is simply $i\pi$.
\end{lemma}
\smallskip

\begin{proof}
We first note that the primitive of the integrand in question involves the complex logarithm, given by log$\left( \tau - \tau_0\right)$. To be able to perform the computation, we must outline how a branch of the logarithm should be chosen. The general form will be to cut along a curve that connects $\tau_0$ and $\infty$, on the right side with respect to the direction of travel of the contour $C$. Moreover, let the curve wander in such a way that it never passes through the contour again. This is clearly possible if the contour is open, and in the case that it is closed, our choice of cutting on the right hand side ensures the existence of such a cut.

\smallskip

Let $C_r$ denotes the contour without the problematic region as in the definition, and let $z_1$ and $z_2$ denote the newly created endpoints of the contour $C_r$ that result from the removal of the neighbourhood around $\tau_0$. We compute:

\begin{align*}
\int_{C_r} \frac{d\tau}{\tau - \tau_0} & = \log \left( \tau - \tau_0 \right) \Big|_{a}^{z_1}+ \log \left( \tau - \tau_0 \right) \Big|_{z_2}^{b} \\
& = \log(b - \tau_0) - \log(a - \tau_0) + \log(z_1 -\tau_0) - \log(z_2 -\tau_0)\\
&= \log(b - \tau_0) - \log(a - \tau_0) + \ln \left( \frac{\labs z_1 - \tau_0 \rabs}{\labs z_2 - \tau_0 \rabs} \right) +\\& \quad\,\, i \left[ \arg(z_1-\tau_0) - \arg(z_2-\tau_0) \right].
\end{align*}

Notice that the values of $\labs z_1-\tau_0 \rabs$ and $\labs z_2 -\tau_0 \rabs$ are equal, which in turn means the corresponding $\ln$ term vanishes. Also, if $C$ is a closed contour, then $a=b$, so the first two $\log$ terms cancel each other out.

\smallskip

Our concerns now lie with what happens to the difference of the arguments in the limit. This proof outline relies on some geometric intuition, and can be replaced with the following line of reasoning: As the contour is smooth, it has a well-defined linear approximation at $\tau_0$. Therefore as we decrease the radius of the open ball around $\tau_0$, the points $z_1$ and $z_2$ are tending towards that linear approximation. Therefore, the angle between the two tends to $\pi$ in the limit. This is certainly the more direct and elegant line of reasoning, but we shall give a more in-depth outline in the following paragraph for the cautious reader.

\smallskip

For simplicity, let us assume that $\tau_0 = 0$, and that on a neighbourhood of $\tau_0$, the contour $C$ lies entirely in the 1st and 3rd quadrants of the coordinate plane, and the tangent to $C$ at each point in this neighbourhood is bounded. By way of a geometric argument, the limit we wish to take is invariant under translation and rotation, thus this same argument may be applied for general contours. Note that the quadrant condition makes use of the fact that $C$ is continuously differentiable on a neighbourhood of $\tau_0$. Let $\gamma(t) = x(t) + iy(t)$ parameterize the contour $C$ such that $\gamma(0)= \tau_0 = 0$, and for $t>0$, $\gamma(t)$ is in the 1st quadrant. Let $\theta'$ be the angle that the tangent vector to $C$ at the origin makes with the positive real axis. Using continuity of the $\arctan$ function and L'Hopital's rule:
\begin{align*}
\lim_{t \to 0^+}\arg{\gamma(t)} & =
\lim_{t \to 0^+}\arctan\left( \frac{y(t)}{x(t)} \right) \\
& = \arctan\left( \lim_{t \to 0^+} \frac{y(t)}{x(t)} \right)\\
& = \arctan\left( \frac{y'(0)}{x'(0)} \right)\\
& = \theta' .
\end{align*}

Now, as $r \to 0$, both $z_1 \to \tau_0$ and $z_2 \to \tau_0$. Since $z_1$ lies in the 3rd quadrant, $\arg(z_1) = \pi + \arg(-\overline{z_1})$. By above, this implies that as $r \to 0$, $\arg(z_1)$ approaches $\pi + \theta'$ if we substitute in $-\overline{\gamma}$. As $z_2$ lies in quadrant 1, it follows directly from the limit calculation that $\arg(z_2)$ approaches $\theta'$ as $r \to 0$. This shows that
\smallskip
\[
\lim_{r \to 0} \left[ \arg(z_1-\tau_0) - \arg(z_2-\tau_0) \right] = \pi
\]

\noindent
whence
\smallskip
\[
\lim_{r \to 0^+} \int_{C_r} \frac{d\tau}{\tau-\tau_0} = \log(b-\tau_0) - \log(a-\tau_0) + \pi i .
\]

\end{proof}

We are now in a position to give sufficient conditions for the existence of the principal value of the singular integral with Cauchy kernel.

\begin{theorem}
\label{CTI existence}
Let $C$ be a smooth contour in the $\tau$-plane. Suppose that the function $\varphi: C \to \mathbb{C}$ is Hölder continuous on $C$. Then for any fixed point $\tau_0$ on the contour not coinciding with the endpoints of $C$, the principle value of the singular integral
\smallskip
\[
\int_{C} \frac{\varphi (\tau)}{\tau -\tau_0}\,d\tau
\]

\noindent
exists. If $C$ is an open contour with true endpoints $a$ and $b$, we can represent the principal value as
\smallskip
\[
P.V.\int_{C}\frac{\varphi (\tau)}{\tau -\tau_0}\,d\tau = P.V.\int_{C} \frac{\varphi (\tau) - \varphi(\tau_0)}{\tau -\tau_0}\,d\tau + \varphi(\tau_0)\left[ \log(b-\tau_0) - \log(a-\tau_0)+\pi i \right] .
\]

On the other hand, if $C$ is a closed contour, we can represent the principal value as 
\smallskip
\[
P.V.\int_{C} \frac{\varphi (\tau)}{\tau -\tau_0}\,d\tau = P.V.\int_{C} \frac{\varphi (\tau) - \varphi(\tau_0)}{\tau -\tau_0}\,d\tau + i\pi\varphi(\tau_0).
\]
\end{theorem}

\begin{proof}
Let $\lambda$ denote the Hölder index of $\varphi$ on $C$. It is clear that we may write the singular integral in the statement of the theorem as
\smallskip
\[
P.V.\int_{C} \frac{\varphi (\tau)}{\tau -\tau_0}\,d\tau = P.V.\int_{C} \frac{\varphi (\tau) - \varphi(\tau_0)}{\tau -\tau_0}\,d\tau + \varphi(\tau_0) \cdot P.V.\int_{C}  \frac{d\tau}{\tau-\tau_0} .
\]

By Lemma $\ref{ipi_integral}$, the second singular integral on the right hand side exists and has the anticipated values for both cases. All that is left to show is to show that the remaining singular integral on the right hand side involving the difference quotient exists in the sense of the principal value. Let $\alpha: [0,1] \to \mathbb{C}$ be a parameterization of the contour $C$, and let $t_0 \in (0,1)$ denote $\alpha^{-1}(\tau_0)$. Expanding this principal value integral into an improper integral of a real variable, we have

\begin{align*}
P.V.\int_{C} \frac{\varphi (\tau) - \varphi(z)}{\tau -\tau_0}\,d\tau & = \int_{0}^{1} \frac{\varphi(\alpha(t))-\varphi(\alpha(t_0))}{\alpha(t)-\alpha(t_0)}\cdot\alpha'(t)\,dt \\
& = \int_{0}^{1} \frac{\varphi(\alpha(t))-\varphi(\alpha(t_0))}{t-t_0}\cdot\frac{t-t_0}{\alpha(t)-\alpha(t_0)}\cdot\alpha'(t)\,dt .
\end{align*}

It is clear that this integral converges away from the point $t_0$, so we only need to check convergence of the integral over some a neighbourhood of $t_0$. Choose $0 < x_1 < t_0$ and $t_0 < x_2 < 1$, and let $I = [x_1,x_2]$. We have reduced the problem to showing that the improper integral

\[
\int_{x_1}^{x_2} \frac{\varphi(\alpha(t))-\varphi(\alpha(t_0))}{t-t_0}\cdot\frac{t-t_0}{\alpha(t)-\alpha(t_0)}\cdot\alpha'(t)\,dt \tag{1}
\]

\noindent
exists. Define a function $f:I \to \mathbb{C}$ by

\[
f(t) = 
\begin{cases}
\frac{t-t_0}{\alpha(t)-\alpha(t_0)}\alpha'(t) & \text{if} \;\,t \neq t_0 \\
1 & \text{if} \;\, t = t_0
\end{cases}.
\]

This function is continuous on $I$, so it follows that there is $M>0$ with the property that $\labs f \rabs \leq M$ on $I$. By only considering the behaviour of $\alpha$ on $I$, it is immediate that $\alpha'$ is continuous on $I$ since we are on a strictly smaller neighbourhood of $t_0$, and hence it is uniformly bounded on $I$. By Theorem $\ref{paramLip}$, $\alpha$ is Lipschitz continuous on $I$, and by Theorem $\ref{holder_composition}$ the composition $\varphi \circ \alpha$ is $\lambda$-Hölder on $I$. Letting $A > 0$ denote the Hölder index of this composition, we have the following bound for the integrand of (1):
\[
\labs\frac{\varphi(\alpha(t))-\varphi(\alpha(t_0))}{t-t_0}\cdot\frac{t-t_0}{\alpha(t)-\alpha(t_0)}\cdot\alpha'(t) \rabs \leq \frac{AM}{\labs t-t_0\rabs^{1-\lambda}} .
\]

When integrated over $I$, the right hand side of the inequality forms a convergent $p$-integral, so by the comparison test for integrals on the real line, the integral in (1) exists, which completes the proof.
\end{proof}

\begin{remark}
Notice that use of the Hölder condition on the density $\varphi$ was only used locally (and often is). That is, for existence of the principle value integral at any given point, the density need only be locally Hölder at that particular point (and of course being integrable everywhere else on the contour). This will continue to be the case for upcoming results.
\end{remark}

\smallskip

We can now extend the domain of definition of the Cauchy type integral, provided the density function behaves accordingly.

\smallskip

\begin{definition} [Value of the Cauchy Type Integral on the Singular Line]
Suppose that the density $\varphi$ of the Cauchy type integral $\Phi$ is locally Hölder at a point $\tau_0$ on the contour of integration. Then we define the value of $\Phi$ at $\tau_0$ as:
\smallskip
\[
\Phi(\tau_0) \coloneqq P.V.\:\frac{1}{2\pi i}\int_C \frac{\varphi(\tau)}{\tau - \tau_0} \,d\tau .
\]
\end{definition}

\section{The Jump Formula}
\label{theJumpFormula}

With a working definition of the Cauchy type integral on the contour of integration, we can begin an analysis of the behaviour of the function near the singular line. More specifically, we are interested in how the original integral definition off the contour and the principal value definition on the contour interact with each other when taking limits to the singular line. This section will accomplish this in two main steps.

\smallskip

The following is merely a convenient technicality.

\begin{proposition}
\label{convinient parameterization}
For any fixed point $\tau_0$ on a contour $C$ that is not one of its endpoints, and any $x_0 > 0$, we can parameterize $C$ with a function $\gamma:[-x_0,x_0] \to C$ so that $\gamma(0) = \tau_0$.
\end{proposition}

\begin{proof}
Suppose that $\alpha: [0,1] \to C$ parameterizes a contour $C$. Note that the interval $[0,1]$ can always be used as the domain of any parameterization since the map $t \mapsto (b-a)t + a$ provides a bijection between $[0,1]$ and an arbitrary closed interval $[a,b]$. Write $t_0 \in (0,1)$ to represent $\alpha^{-1}(\tau_0)$. For $x_0 > 0$, we can define $\gamma$ by:
\smallskip
\[
\gamma(t) = 
\begin{cases}
\alpha\left(t_0\left[\frac{t}{x_0}+1\right]\right) &\! \text{if} \; -x_0\leq t\leq 0 \\[7pt]

\alpha\left( \frac{t}{x_0}[1-t_0]+t_0 \right) &\! \text{if} \;\: 0<t\leq x_0 .
\end{cases}
\]
\end{proof}

As is often the case, we use a Lemma for the heavy lifting. Known as ``The Basic Lemma," the result shows that the same density condition as in Theorem \ref{CTI existence} is enough for a closely related integral-defined function to be continuous.

\begin{lemma}
\label{basic lemma}
Let $C$ be a smooth contour in the $\tau$-plane. Let $\tau_0$ be a point on $C$ that does not coincide with an endpoint of $C$. Suppose that the function $\varphi: C \to \mathbb{C}$ satisfies the Hölder condition. Define $\psi: \mathbb{C} \to \mathbb{C}$ by
\smallskip
\[
\psi(z) = \int_{C} \frac{\varphi(\tau) - \varphi(\tau_0)}{\tau - z} \,d\tau .
\]

Then $\psi$ is a well-defined function, and is continuous at the point $z=\tau_0$. That is,
\[
\lim_{z \to \tau_0} \psi(z) = \psi(\tau_0) = P.V.\int_C \frac{\varphi(\tau) - \varphi(\tau_0)}{\tau - \tau_0} \,d\tau .
\]
\end{lemma}

\begin{proof}
Note that $\psi$ being well-defined is a consequence of Theorem \ref{CTI existence}. All we need to show now is that it is indeed continuous at $\tau_0$. Write
\smallskip
\[
\labs\psi(z)-\psi(\tau_0)\rabs = \blabs P.V.\int_{C} \frac{z-\tau_0}{\tau-z} \cdot\frac{\varphi(\tau) - \varphi(\tau_0)}{\tau-\tau_0}\,d\tau \brabs .\tag{1}
\]

Let $\delta>0$ denote the radius of the disk centered at $\tau_0$ that contains exactly one arc of the contour $C$. We denote this arc by $C_{\delta}$. In general, we will have to consider smaller neighbourhoods around $\tau_0$. To avoid confusion in notation, let us agree that the value of $\delta$ may be reset to smaller values as needed. As the singularity is contained on the arc $C_{\delta}$, it will be useful to split the integral in (1) into two pieces: one over this arc, the other over the rest of the contour $C-C_{\delta}$.
Let $A > 0$ and $0<\lambda\leq1$ denote the Hölder constant and index of $\varphi$ on the contour $C$, respectively.

To bound the first term of the integrand in (1) over $C_{\delta}$, we need to consider two cases. First, we will examine the non-tangential limit. Consider the triangle with vertices $z$, $\tau$ and $\tau_0$. Label the angles at vertices $\tau$ and $\tau_0$ by $\theta_{\tau}$ and $\theta_{\tau_0}$, respectively. By the law of sines, we have the relationship:
\smallskip
\[
\frac{\sin(\theta_{\tau_0})}{\labs \tau-z \rabs} = \frac{\sin(\theta_{\tau})}{\labs z-\tau_0\rabs} .
\]

\smallskip
As the path $z$ approaches $\tau_0$ is not tangent to the contour $C$, there is a $\delta$-neighbourhood of $\tau_0$ in which the angle $\theta_{\tau_0}$ is bounded away from both 0 and $\pi$. Therefore, there is $\omega_0 > 0$ with the property that $\sin(\omega_0) \leq \sin(\theta_{\tau_0})$ in this neighbourhood of $\tau_0$. We can then derive the bound
\smallskip
\[
\labs \frac{z-\tau_0}{\tau-z}\rabs \leq \frac{1}{\sin(\omega_0)} \tag{2}.
\]

Using Proposition \ref{convinient parameterization}, let $\alpha: [-1,1] \to C_{\delta}$ be a parameterization of the arc where $\alpha(0) = \tau_0$ and the endpoints of the domain map to the endpoints of the arc. If we restrict the image of $\alpha$ by shrinking $\delta$ from its current value, we also shrink the domain of the parameterization. So for $0<\delta'\leq\delta$, the restriction is given by $\alpha_{\delta'}:[l(\delta'), u(\delta')] \to C_{\delta'}$, where the functions $l$ and $u$ satisfy $-1 \leq l(\delta') < 0$ and $0 < u(\delta') \leq 1$, both functions map to the endpoints of the arc $C_{\delta'}$ and both tend to zero as $\delta'$ tends to zero.

Let $I_1$ denote the integral in (1) taken over the arc $C_{\delta}$. Applying (2) and the usual comparison theorem for contour integrals:
\smallskip
\begin{align*}
\labs I_1 \rabs = \blabs P.V.\int_{C_{\delta}} \frac{z-\tau_0}{\tau-z} \cdot\frac{\varphi(\tau) - \varphi(\tau_0)}{\tau-\tau_0}\,d\tau\brabs
& \leq \frac{1}{\sin(\omega_0)} \int_{l(\delta)}^{u(\delta)} \labs\frac{\varphi(\alpha_{\delta}(t)) - \varphi(\alpha_{\delta}(0))}{\alpha_{\delta}(t) - \alpha_{\delta}(0)}\rabs \,dt .
\end{align*}
\smallskip

The integral on the right-hand side of the inequality can be bounded above by a convergent \emph{p}-integral as was demonstrated in the proof of Theorem \ref{CTI existence}. Using the same notation for constants, we have
\begin{align*}
\frac{1}{\sin(\omega_0)} \int_{l(\delta)}^{u(\delta)} \labs\frac{\varphi(\alpha_{\delta}(t)) - \varphi(\alpha_{\delta}(0))}{\alpha_{\delta}(t) - \alpha_{\delta}(0)}\rabs \,dt &\leq \frac{AM}{\sin(\omega_0)} \int_{l(\delta)}^{u(\delta)} \frac{dt}{\labs t \rabs^{1-\lambda}} .
\end{align*}

Finally note that since we may need to shrink $\delta$, the values of $l$ and $u$ may no longer be symmetric about 0, so we shall take $r(\delta)$ to be the larger of the two in terms of magnitude. Thus
\begin{align*}
\labs I_1 \rabs & \leq \frac{AM}{\sin(\omega_0)} \int_{-r(\delta)}^{r(\delta)} \frac{dt}{\labs t \rabs^{1-\lambda}} \\
& = \frac{2AM}{\sin(\omega_0)} \lim_{x \to 0^{+}} \int_{x}^{r(\delta)} \frac{dt}{t^{1-\lambda}}\\
& = \frac{2AMr(\delta)^{\lambda}}{\lambda\sin(\omega_0)}.\tag{3}
\end{align*}

Let $\varepsilon > 0$ be given. By all previous discourse, we can choose $\delta > 0$ such that $\labs I_1 \rabs < \varepsilon/2$ . Let $I_2$ denote the integral in (2) taken over the contour(s) $C-C_{\delta}$. Notice that the integrand of $I_2$ does not possess any singularities, so it is indeed continuous at the point $z$. That is, there is $\delta$\textemdash no larger than the previously established value\textemdash satisfying $\labs I_2 \rabs < \varepsilon/2$, for $z$ $\delta$-close to $\tau$. Alas, we obtain
\smallskip
\[
\labs \psi(z) - \psi(\tau_0) \rabs \leq \labs I_1 \rabs + \labs I_2 \rabs < \varepsilon 
\]
\smallskip
\noindent
for $\labs z - \tau_0 \rabs < \delta$. This takes care of the non-tangential case.

Next, we show that $\psi$ behaves as a continuous function at $\tau_0$ when the path of approach is along contour itself. Notice that in our analysis thus far, our estimates have been independent of the point $\tau_0$ on the contour. That is, for $z$ approaching $\tau_0$ on either side of the contour, $\psi$ behaves as a uniformly continuous function.

\smallskip

Let $\varepsilon > 0$ be given. Then for $z$ not on the contour, there exists a $\delta > 0 $ such that if $\labs z -\tau_0 \rabs < \delta$, then $\labs \psi(z) - \psi(\tau_0) \rabs < \varepsilon/2$. Here, we are assuming $z$ is taking a non-tangential approach. Choose a point $\tau_1$ on the contour $C$ such that $\labs \tau_0 - \tau_1 \rabs < \delta/2$ and a point $z$ on a fixed side of $C$ so that $\labs z-\tau_0 \rabs < \delta/2$. Then $\labs z -\tau_1\rabs < \delta$, and in view of the uniformity of non-tangential approaches mentioned previously, we have $\labs \psi(z) - \psi(\tau_1) \rabs < \varepsilon/2$. Thus
\smallskip
\[
\labs \psi(\tau_0) - \psi(\tau_1) \rabs \leq \labs \psi(\tau_0) - \psi(z) \rabs + \labs \psi(z) - \psi(\tau_1) \rabs
<\varepsilon
\]

\noindent
proving continuity for limiting values along the contour.

\smallskip

Finally, we look at the tangential limit case. Both of the previously proven cases can be used. On page 34 in ``Singular Integral Equations" by Muskhelishvili \cite{muskhelishvili2008singular}, it is shown that we can always find a point $\tau'$ on the contour $C$ such that the quantities $\labs z - \tau'\rabs$ and $\labs z - \tau_0 \rabs$ are both arbitrarily small, where $\tau'$ approaches $\tau_0$ along a non-tangential path. The proof is quite technical, and best left to the curious reader to examine its details. This result combined with the first two cases guarantees that
\smallskip
\[
\labs \psi(z) - \psi(\tau_0) \rabs \leq \labs \psi(z) - \psi(\tau') \rabs + \labs \psi(\tau') - \psi(\tau_0) \rabs
<\varepsilon
\]

\noindent
whenever $z$ is sufficiently close to $\tau_0$. This concludes the proof of the lemma.
\end{proof}

\smallskip

With all the difficult work behind us, the main result follows relatively easily, although what it reveals is rather unintuitive.

\smallskip

\begin{theorem}[The Jump Formula on Smooth Contours]
Let $C$ be a smooth contour in the $\tau$-plane. If the density $\varphi$ of the Cauchy type integral
\smallskip
\[
\Phi (z) = \frac{1}{2\pi i} \int_{C} \frac{\varphi(\tau)}{\tau - z}\,d\tau
\]

\noindent
satisfies the Hölder condition on $C$, then both $\Phi^{+}$ and $\Phi^{-}$ have limiting values on approach to a fixed point $\tau_0$ on the contour (not coinciding with its endpoints), which are given by the following formulae:
\smallskip
\[
\Phi^{\pm}(\tau_0) = \pm\frac{1}{2}\varphi(\tau_0)+\Phi(\tau_0). 
\]

This equation is referred to simply as the Jump formula, or sometimes as the Sokhotski formulae.
\end{theorem}

\begin{proof}
We may begin by assuming that the contour $C$ is closed. If the given contour happens to be open, we attach an open smooth contour $C'$ to the end points of $C$ so that a smooth closed contour is formed, and then define $\varphi$ to be zero on $C'$. 
Fix a point $\tau_0$ on the contour $C$ as described in the statement of the theorem. By Theorem \ref{CTI existence}, $\Phi$ exists in the sense of the Cauchy principal value. We will use the notation $\Phi^{+}(\tau_0)$ and $\Phi^{-}(\tau_0)$ as shorthand for a limit of $\Phi$ taken from inside $D^{+}$ and $D^{-}$ on approach to $\tau_0$, respectively.

First, recall the possible values of the following singular integral:
\smallskip
\[
\int_{C} \frac{d\tau}{\tau - z} = \begin{cases}
2\pi i &\! \text{if} \;\: z \in D^{+}\\
0 &\! \text{if} \;\: z \in D^{-}\\
\pi i &\! \text{if} \;\: z \in C .
\end{cases}
\]

\noindent
Note that the value of the integral for $z$ located on the contour of integration is principal and was computed in Lemma $\ref{ipi_integral}$. The fact that this integral defines a step function in $z$ is key. Consider the following integral-defined function on $\mathbb{C}$:
\smallskip
\[
\psi(z) \coloneqq \frac{1}{2\pi i} \int_{C} \frac{\varphi(\tau) - \varphi(\tau_0)}{\tau - z} \,d\tau = \frac{1}{2\pi i} \int_{C} \frac{\varphi(z)}{\tau-z}\,d\tau -\varphi(\tau_0)\left[\frac{1}{2\pi i} \int_{C} \frac{d\tau}{\tau - z}\right] .
\]

From the Basic Lemma \ref{basic lemma}, we know that this function $\psi$ is continuous at $\tau_0$, whence $\psi(\tau_0) = \psi^{+}(\tau_0) = \psi^{-}(\tau_0)$. If we look at the limits as $z$ tends to $\tau_0$ from each of these three distinct regions, we get the following three equations:

\begin{align*}
\Rightarrow \lim_{\substack{z \to \tau_0 \\ z \in D^{+}}} \psi(z) &= \lim_{\substack{z \to \tau_0 \\ z \in D^{+}}} \Biggr[ \frac{1}{2\pi i} \int_{C} \frac{\varphi(\tau)}{\tau-z}\,d\tau -\varphi(\tau_0)\left[\frac{1}{2\pi i} \int_{C} \frac{d\tau}{\tau - z}\right] \Biggr] \\
& = \Phi^{+}(\tau_0) - \varphi(\tau_0) \tag{1}\\[1em]
\Rightarrow \lim_{\substack{z \to \tau_0 \\ z \in D^{-}}} \psi(z)  &= \lim_{\substack{z \to \tau_0 \\ z \in D^{-}}} \Biggr[ \frac{1}{2\pi i} \int_{C} \frac{\varphi(\tau)}{\tau-z}\,d\tau -\varphi(\tau_0)\left[\frac{1}{2\pi i} \int_{C} \frac{d\tau}{\tau - z}\right] \Biggr] \\
& = \Phi^{-}(\tau_0) \tag{2}\\[1em]
\Rightarrow \psi(\tau_0) & = P.V.\:\frac{1}{2\pi i} \int_{C} \frac{\varphi(\tau)}{\tau - \tau_0}\,d\tau -\varphi(\tau_0)\left[P.V.\:\frac{1}{2\pi i} \int_{C} \frac{d\tau}{\tau - \tau_0}\right] \\
& = \Phi(\tau_0) -\frac{1}{2}\varphi(\tau_0) .\tag{3}
\end{align*}

\noindent
By the proceeding remarks, the expressions (1), (2), and (3) are equal:
\smallskip
\[
\Phi^{+}(\tau_0) - \varphi(\tau_0) = \Phi^{-}(\tau_0) = \Phi(\tau_0) -\frac{1}{2}\varphi(\tau_0).
\]

\noindent
Solving for both of $\Phi^{\pm}$ yields the announced formula.
\end{proof}
\smallskip

\begin{remark}
The Jump formula can be rearranged to obtain the following identity over $C$ minus the endpoints:
\smallskip
\[
\varphi(\tau) = \Phi(\tau)^+ - \Phi(\tau)^- .
\]

\noindent
So we see that a Hölder continuous function defined on a smooth contour $C$ can be written as the difference of the boundary values of two functions which are each analytic on one side of $C$.
\end{remark}

\medskip

Let us return to the version of the Cauchy integral formula we proved near the beginning of the chapter. We can begin to justify the lengths went to that merely weakened the function conditions on the boundary. The reader may have been unconvinced that this result was all that necessary. After all, it is non-trivial to find a function that is analytic in some domain, yet \emph{only} continuous on the boundary of that domain (as a side note, the Cauchy type integral with Hölder density is in general one of these functions). Here, we concern ourselves with the the statement of the theorem, and what it says about the Jump formula. Portions of the next chapter will promote the importance of the technique used in its proof.

\smallskip

Let $f^{\pm}$ denote analytic functions in $D^{\pm}$, respectively, and let us assume that $f^{-}(\infty) = 0$. Suppose that both functions have a continuous extension to the contour $C$. Defining a function $\varphi$ on $C$ as the difference of the extensions, we can write:
\smallskip
\begin{align*}
\frac{1}{2\pi i}\int_{C} \frac{\varphi(\tau)}{\tau -z}\,d\tau &= \frac{1}{2\pi i}\int_{C} \frac{f^{+}(\tau)}{\tau -z}\,d\tau - \frac{1}{2\pi i}\int_{C} \frac{f^{-}(\tau)}{\tau -z}\,d\tau \tag{1}.
\end{align*}

Notice that the functions $f^+$ and $f^-$ fit the requirements of Theorems' \ref{CIF} and \ref{CIF analytic in D-}, respectively. Thus we obtain the two valuations:
\smallskip
\begin{align*}
\frac{1}{2\pi i}\int_{C} \frac{f^{+}(\tau)}{\tau -z}\,d\tau & = \begin{cases}
f^{+}(z) &\! \textup{if} \;\: z \in D^{+}\\
0 &\! \textup{if} \;\: z \in D^{-}
\end{cases} \tag{2}\\[1em]
\frac{1}{2\pi i}\int_{C} \frac{f^{-}(\tau)}{\tau -z}\,d\tau & = \begin{cases}
0 &\! \textup{if} \;\: z \in D^{+}\\
-f^{-}(z) &\! \textup{if} \;\: z \in D^{-} .\tag{3}
\end{cases}
\end{align*}

Using (2) and (3), we can write (1) as
\smallskip
\[
\frac{1}{2\pi i}\int_{C} \frac{\varphi(\tau)}{\tau -z}\,d\tau = \begin{cases}
f^{+}(z) &\! \textup{if} \;\: z \in D^{+}\\
f^{-}(z) &\! \textup{if} \;\: z \in D^{-} .
\end{cases}
\]

This is a version of the Cauchy integral formula that extends to an even more general class of functions on $C$.
An interpretation of the Jump formula is that most ``nice" functions\textemdash which in our case meant Hölder continuous functions\textemdash arise in this way, as the \emph{difference} of analytic functions.

\smallskip

The Cauchy integral formula, in its original formulation, reproduces an analytic function if you start with an analytic function on and inside a closed contour. It also solves the boundary value problem for analytic functions, but only if the boundary values are that of an analytic function. The rigidity of analytic functions means that even most smooth functions are not the boundary values of any analytic function. This means that this boundary value problem is not well-posed. The Jump formula actually fixes the issue by replacing it with a new problem. Given a density function, find two analytic functions whose difference is the density function. As we have shown, if the density function is Hölder continuous, then the Cauchy type integral gives us the analytic functions that reproduce in the same way the Cauchy integral formula does, but now on \emph{both} sides of the contour, and it solves the boundary value problem. The Jump formula says that every Hölder continuous function is a jump of two analytic functions. Boundary value problems and well-posedness will be discussed in the first section of the next chapter.
\chapter{The Research Problem}
\label{chapter2}

In the previous chapter, we studied the Cauchy type integral and found sufficient conditions for its existence, and then using the Jump decomposition we extended the Cauchy integral formula. A solid understanding of the definitions, techniques, and results used throughout chapter 1 are essential for this chapter. This is because the presentation is centred around how they come together to tell a more complete story of this corner of mathematics, and where it is \emph{incomplete}. 

\smallskip

In stark contrast to chapter 1, our discussion will be carried out in a very mathematically informal matter. What this means is that ideas will be greatly simplified, and details will be sparse, perhaps non-existent at times. Of course, these topics and their genealogy take years of training to thoroughly understand. The goal of this chapter is just to give the reader some appreciation for \emph{why} we care in the first place.

\section{Boundary Value Problems}

The term ``Boundary value problem" is used to describe problems whose solutions are functions satisfying a particular differential equation, and have the added constraint that they must agree with predefined values on the boundary of the domain. The book ``Boundary Value Problems" \cite{gakhov2014boundary}\textemdash in which chapter 1 takes its main inspiration from\textemdash is concerned with solutions to this class of problems. All of the material covered in chapter 1 is used to begin forming a basis of knowledge for understanding boundary value problems\textemdash and their solutions\textemdash in general.
\smallskip

We need to establish some jargon before further discussion. A particular boundary value problem is said to be \emph{well-posed} if a solution exists, that solution is unique, and the solution depends continuously on the initial values. If a question is not well-posed, then either modifications to the conditions, or a new question entirely is needed to arrive at a problem that \emph{is}. Well-posed problems are of particular importance to both mathematicians and scientists. First, having a unique solution is a must as it indicates that there is enough data present for the problem to be accurately captured by the model. Second, the solution depending continuously on the data ensures \emph{stability}. In reality, measurements always have some margin of error, and we would expect that a small change in the measurement should not drastically change what the solution tells us. 

Note that this is only the minimum requirements defined for well-posedness of general boundary value problems. Different disciplines and even particular problems may require further, stronger conditions.

\subsection{The Dirichlet Problem}
Let us focus on a specific example of a boundary value problem that also happens to be one of the most important. It is known as the Dirichlet problem.

\smallskip

Let $D$ be a domain in the finite plane whose boundary is a Jordan curve (called a \emph{Jordan domain}), and let $u: \partial D \to \mathbb{F}$ be a continuous function on the boundary of $D$, where the symbol $\mathbb{F}$ is used to denote either the field of real or complex numbers. 

The Dirichlet problem states: Is there a uniquely-determined function $f: \overline{D} \to \mathbb{F}$ satisfying the following properties:

\begin{enumerate}[i]
    \item\!\!. $f$ is continuous.
    \item\!\!. $\Delta f = 0$ on $D$. That is, $f$ is a harmonic function on $D$.
    \item\!\!.  $f = u$ on the boundary of $D$.
\end{enumerate}

The problem is in fact well-posed. There are specific, interesting examples for certain functions $u$ and certain domains, but the general formula for $f$ is given by
\smallskip
\[
f(z) = - \frac{1}{2\pi}\int_{\partial D} u(w) \frac{\partial g}{\partial n_w} (w,z) \,ds_w
\]

\noindent
where $n_w$ is the unit outer normal with respect to $w$, meaning it is the directional derivative of $g$ in the direction of the normal vector to the curve, $ds_w$ means we are integrating the variable $w$ with respect to arc length, and $g$ is the Green's function of the Laplacian operator. The details of this general solution is not needed in future discussions, and is merely stated for reference.

\subsection{The Holomorphic Boundary Value Problem}

Next, we look a specific example of a boundary value problem whose solution is connected to the Jump problem. 

\smallskip

Again, let $D$ be a Jordan domain. The holomorphic boundary value problem poses: Given a function $u: \partial D \to \mathbb{C}$ that is continuous (or possibly, with even ``nicer" requirements), does there exist a holomorphic function that satisfies the following:

\begin{enumerate}[i]
    \item\!\!. $f$ extends to the boundary of $D$
    \item\!\!. $f$ is holomorphic on $D$
    \item\!\!. $f = u$ on the boundary of $D$.
\end{enumerate}

This problem is \textbf{not} well-posed! It only works for \emph{some} $u$, indicating that we haven’t stated the right problem yet. For example, the function $u(z)=1/z$ is in particular continuous on the boundary of the unit disk centered at the origin. However, since holomorphic functions are uniquely determined, it follows that the function $f$ must equal $u$, but then $f$ is not holomorphic on $\mathbb{D}$.

\medskip

Let us impose a stricter condition on $u$ and see what can be said. Assume that $u$ is not only continuous on the boundary, but \emph{analytic}. As analyticity is a property that is defined in open sets, $u$ has an analytic extension to a neighbourhood of the boundary. To keep things simple, assume that the boundary of $D$ is analytic, that is, it can be parameterized by an analytic function on the unit circle (See Appendix \ref{Appendix 2}). If there were a solution to this holomorphic boundary value problem, it would first need to be an analytic continuation of $u$ to the domain $D$. Since analytic continuations are unique, the function $f$ must be given by the Cauchy type integral:
\smallskip
\[
f(z)=\frac{1}{2\pi i}\int_{\partial D} \frac{u(\tau)}{\tau - z}\,d\tau . \tag{1}
\]

Why? The Cauchy integral formula tells us that for points sufficiently close to the boundary, $f = u$, where $u$ is now referring to the analytic extension of $u$ to the neighbourhood. Let us denote $f^{\pm}$ as we did in the previous chapter, where the plus-minus denotes the interior and exterior domains, respectively. The function $f$ we defined above is therefore equal to $f^{+}$. By the Jump decomposition, we know that $u = f^{+} - f^{-}$, and finally to satisfy the third requirement of the problem, we must have $f^{+} = f = u$, whence $f^{-} \equiv 0$.

\smallskip

To summarize, for the holomorphic boundary value problem to be well-posed with data function $u$, the solution must have the property that $f^{-} \equiv 0$. Not only is this a \emph{necessary} condition, it is also \emph{sufficient}. That is, there exists a solution $f$ if and only if $f^{-}\equiv 0$.

\smallskip

We have already developed the tools to demonstrate this! Suppose that $f^{-} \equiv 0$. By Corollary \ref{Analytic fn are lipschitz on smooth contours} and the Jump formula, we know that 
\smallskip
\[
u = f^+ - f^-
\]

\noindent
on the boundary of $D$. Note that both functions on the right-hand side are analytic by Theorem \ref{CTI is analytic}, and can be analytically continued to the boundary. We can justify this by moving the contour slightly outward\textemdash staying in the region of holomorphicity of $u$\textemdash using the principal of deformations of paths. The assumption that $f^- \equiv 0$ forces the relation $u = f^+$ on the boundary. This shows that $f^+$ solves the holomorphic boundary value problem (See Appendix \ref{Appendix 0} for details regarding the Jump problem for holomorphic density functions).

\smallskip

Returning to our example, observe that when $u(z) = 1/z$, a quick computation reveals that $f(2)$ is nonzero, so $f^- \not\equiv 0$.

\medskip

Notice that all the work we have done to understand the Cauchy integral and jump formula gives us a clear understanding of when the holomorphic boundary value problem is well-posed. This is stated as follows:
\smallskip
\begin{center}
\emph{Given an analytic function u on the boundary of D, does there exist a pair of functions $f^{\pm}$ holomorphic on $D^{\pm}$ such that $u = f^+ - f^-$ on the boundary of $D$?}
\end{center}
\smallskip

We now know that the answer is yes. The analytic conditions imposed on $u$ and the boundary of $D$ are flexible as well. For example, we could assume that the boundary of $D$ is smooth, and $u$ is Hölder continuous, as we did in the previous chapter. By doing this, we ended up solving a stronger version on the problem. The weaker version we talked about just now was if the boundary of $D$ is analytic, and $u$ is analytic on a neighbourhood of the boundary. We have now seen that both are well-posed problems.
 
\smallskip
 
The Jump problem is also connected to the Dirichlet problem. It was Fredholm who used the solution to a related jump problem that concerns itself with real function to solve the Dirichlet problem.


\section{Quasicircles}
\label{quasicircle}

In the previous chapter, we only concerned ourselves with \emph{smooth} contours. That is, contours that have a continuously differentiable parameterization, and non-vanishing tangent vectors. These were our examples of “nice” curves in the plane. Naturally, we may ask ourselves: “What is the \emph{worst} curve we can develop a Jump formula for?” By worst, we mean the most poorly-behaved curves, or least \emph{regular} curves, in some well-defined sense. This is where the quasicircle comes in. The technical details for the theory of quasicircles is far too advanced for this article, but an attempt will be made to give an intuitive picture to the reader. Let us begin the motivation by first thinking about a \emph{goal} of mathematics in general.

\smallskip

Very broadly, every field in mathematics has a basic \emph{object} that attempts to capture an \emph{interesting} property, or physical phenomena. Mathematicians then collectively work out the details to converge on a set of axioms that form the foundation for the objects and the field itself. For example, group theory is about understanding the minimum \emph{structure} needed to talk about \emph{symmetry}. The object is the group itself, and it is the group axioms combined with a binary operation that captures symmetry is an abstract way. Similarly, topology is about understanding the minimum structure needed to talk about continuity. The objects are a family of ``open" sets that form a space, and we say that a function between spaces is continuous when its preimage of open sets are themselves open sets. Complex analysis is no different at its core. It is about understanding the minimum structure needed to talk about complex differentiable maps, which we often refer to interchangeably as \emph{analytic} or \emph{holomorphic} functions (although typically one is chosen to emphasise the particular equivalent property we are thinking of the functions having). The basic objects in complex analysis are called Riemann surfaces. So far, we have only worked with two very basic examples of these objects: The plane $\mathbb{C}$, and the plane plus the point at infinity, which is called the Riemann sphere. One way of viewing Riemann surfaces is as sets equipped with a notion of \emph{angle}. Note that one-to-one holomorphic functions, called \emph{conformal} maps, are functions that are locally angle-preserving. In a sense, holomorphic maps are to Riemann surfaces as homeomorphisms are to topological spaces, or group homomorphisms are to groups, or linear transformations are to vector spaces. All of these are examples of structure preserving maps, and are central to talking about and comparing instances of the basic objects in each respective field. There is a disparity with complex analysis when compared to other subjects, especially the two given as examples here. It is expected that the reader is not familiar with general Riemann surfaces\textemdash even with enough background to read the previous chapter\textemdash whereas the study of many other fields almost begins with the definition of their basic objects.

\smallskip

We are ready to return to the main topic of the section: quasicircles. Conformal maps, as defined above, have a very satisfying geometric interpretation. Recall that the Jacobian of a differentiable function is a matrix that gives the best linear approximation of the function at a point, and its entries are the partial derivatives of the function. The Jacobian of a surjective conformal map sends circles to circles. The implication here is that a conformal map is locally a rotation and a rescaling. One can show this using the fact that the partial derivative of analytic functions satisfy the Cauchy-Riemann equations. This characterization of the Jacobian ends up being equivalent to the map itself begin angle-preserving. If a map is not conformal, but is continuously real differentiable, bijective, and orientation preserving, then the map sends circles to ellipses. By orientation preserving, we mean that the direction of traversal of the parameterization of any curve in the domain remains the same under the map. An example of a map that does \emph{not} preserve orientation is the map $z \mapsto 1/z$. Under this map, the unit circle has its direction of traversal reversed. This is easily seen when we plug in the standard parameterization of the unit circle into the map: $e^{i\theta} \mapsto e^{-i\theta}$. 

The fact that this not-quite conformal map sends circles to ellipses means that it is not angle-preserving in general. A \emph{quasiconformal} map is roughly a function whose Jacobian sends circles to ellipses, where the ratio between the major and minor axis of the ellipses in the image is globally bounded. These maps effectively deform Riemann surfaces by distorting their angle structure, creating a new Riemann surface. 

\smallskip

A \emph{quasicircle} can be defined as the image of the unit circle under a quasiconformal map from the plane onto itself. The technical definition underlying quasiconformal maps permit quasicircles to be incredibly rough and jagged, like a fractal. Two well known examples of rough curves that are in fact quasicircles are the Koch snowflake, and quadratic Julia sets where the constant term lies in the main component of the Mandelbrot set.

Quasicircles need not be rectifiable, meaning they can \textbf{locally} have infinite length. An intriguing implication of this is that they may bound a finite area with an infinitely long perimeter. It is perhaps far too ambitious to thoroughly cover \emph{why} quasicircles are of so much interest to complex analysis, or even the intersection of subfields we are concerned with. Instead, we will point in the direction of a few reasons here, and expand on the third later. 

\smallskip

First, there is a way that quasicircles can be placed in a one-to-one correspondence with the moduli space of Riemann surface \cite{lehto2012univalent}. The technicalities here are not important, only that a correspondence with a space containing the fundamental objects of complex analysis exists. Second, they show up as the ideal limit in physical processes, including random Brownian motion, percolation, and conformal field theories. Finally (and most importantly to us), many results involving the Cauchy type integral hold if and only if the curve is a quasicircle. Note that ``if and only if" statements are \emph{rare} in analysis when compared to other fields. Often the converse of a theorem has a pathological counterexample, especially in Real Analysis (see \cite{gelbaum2003counterexamples}). Discovering such a correspondence can indicate something significant is happening. Some of the famous results the reader might be familiar with include the Heine-Borel theorem, which characterizes compact sets in $n$-dimensional Euclidean space, the characterization of open sets in $\mathbb{R}$, and the equivalence of convergence and the Cauchy criterion of sequences in $\mathbb{R}$. Notice how each of these fundamental results relate back to completeness and compactness. These properties are \emph{incredibly} important in analysis, and are closely linked to existence and uniqueness. 

\smallskip

Unfortunately, many of the basic formulations of theorems in introductory complex analysis are unable to say anything about quasicircles. So far, the contours we have been integrating over have always consisted of finitely many piecewise smooth arcs, joined end to end. Theorems that we have relied on, like the Cauchy integral formula, no longer make sense with such weak constraints on the curve. A goal of modern research in the field is to solve this problem. We will postpone this portion of the discussion to the end of the chapter. We will briefly mention a work-around currently being used to deal with quasicircles. The general technique we used to prove Theorem \ref{CIF} is employed to get a handle on domains whose boundaries are quasicircles. In the proof, we essentially approximated the curve via homotopy and use a limit to get the result. This was to make up for loss of analyticity on the boundary. For quasicircles, the same general principal is employed, but with a different end goal: Approximate the boundary with more regular curves, and take a limit.

\section{Dirichlet Space}

Differential equations permeate mathematics. They tend to show up in unexpected places, far removed from physics where we tend to think of them originating from. A generic problem in mathematics is to solve an equation of the form $Lu = f$, where $L$ is a differential operator (such as the Laplace operator $\Delta$), $u$ is the function we wish to find, and $f$ is a function roughly representing some condition. The general procedure of constructing solutions to the equation is to show existence and uniqueness of solutions, which in turn is done by constructing the solution with an approximation (say a limit of a sequence of functions). Both existence and uniqueness rely on the ideas of completeness and compactness, so we ask ourselves: what does completeness and compactness \emph{look} like in a space of functions? Attempting to answer this question leads us to the study of functional analysis. Very broadly, functional analysis looks to transforms existence, uniqueness, and approximations into completeness and compactness of function spaces. It does this by studying function spaces using the tools of linear algebra. Different differential equation have different natural function spaces which are well suited for their study. This connection between the two is a current area of research in analysis.

\smallskip

A general class of function spaces very important to mathematics and physics are \emph{Hilbert spaces}. A Hilbert space is a complete inner product space. A space being complete has the same meaning as what we mean when we say the the real or complex numbers are complete: every Cauchy sequence in the space converges in said space. An inner product space is a vector space endowed with an inner product; a way of generalizing the idea of “length” and “orthogonality” to abstract vector spaces. The particular function space we will talk about is an example of a Hilbert space, known as \emph{Dirichlet space}.

\smallskip

Function spaces need a notion of measuring functions so that we can define not only their \emph{size} in the space, but also their \emph{distance} to other functions in the space. We will use a measurement called the \emph{Dirichlet energy}. More precisely if $D$ is a domain on the Riemann sphere, and $f$ is analytic on $D$, we call the quantity
\[
\frac{1}{\pi} \iint\limits_{D-\{\infty\}} |f'|^2 \,dA
\]

\noindent
the Dirichlet energy of $f$, where $dA$ is the Lebesgue area measure. Intuitively, the Dirichlet energy of an injective function on a set measures the area of the image of the function, as in this case $\text{det}\,Df = |f'|^2$. This intuition falls apart if $f$ is not injective, and cannot fully be recovered in general. One could say that in this case the energy measures the area of a multisheeted Riemann surface, but it is possible that there are infinitely many sheets, bad behaving singularities, and so on, which limit how far the intuition will take us.

\smallskip

The \emph{Dirichlet space} $\mathcal{D}(\Omega)$ of an open set $\Omega$ in the complex plane is a function space whose elements are holomorphic functions on $\Omega$ with finite Dirichlet energy. While the informal description of Dirichlet energy is sufficient for our discussion, one technicality needs to be taken care of. Since the Dirichlet energy of a constant function is always 0 (why?), it only defines a semi-norm on the space. We can turn the energy into a norm by imposing a normalizing restriction on the space so that only functions that vanish somewhere are included. In a sense, we don't ``lose" any functions, as we can always add a constant to a function to make it vanish somewhere. The implications of adding such a restriction have not yet been fully explored, but so far it seems to work. The notation to refer to the subspace of functions who vanish at a point $p$ is $\mathcal{D}_p(\Omega)$. 

\smallskip

Another similar function space often employed is the \emph{harmonic} Dirichlet space $\mathcal{D}_{\textup{harm}}(\Omega)$, consisting of the harmonic functions on $\Omega$ with finite Dirichlet energy.

In the harmonic Dirichlet space of a Jordan domain on the Riemann sphere, every element has an extension to the boundary, except on a negligible subset, which is a technical condition of a set having ``zero capacity". For reference, this is an even smaller set than a measure-zero set. The extension is more complicated that simply an “approach”. We consider wedge-shaped approaches that avoid problematic angles (think back to the different approach cases in the proof of Lemma \ref{basic lemma}). Now, if two functions in the space have the same boundary values in the sense defined above, then it turns out these functions must be the same.


\section{Faber Polynomials and Series}
\label{FaberPolynomialDiscussion}

As previously stated, the main goal of complex analysis is to understand complex-differentiable maps. A remarkable fact of these maps is that they may be written as a convergent power series centered at each point of holomorphicity, a condition often referred to as analyticity. The converse is also true. The sum of a complex power series is holomorphic at each point interior to its circle of convergence. In a sense, holomorphic functions behave as “infinite polynomials”, because their series representation can be manipulated as such whenever they converge. This may give another insight as to why we care so much about complex-differentiable maps. Polynomials are the most well-behaved and easy to work with functions.

\smallskip

One of the drawbacks of power series is that they are convergent only on disks. We are forced to derive a new series representation and stitch them together to cover a domain, should that domain be a shape which is not exactly a disk. Even one singularity restricts the convergent power series to the disk with a radius that is the distance from the center to the singularity, even if the function is holomorphic everywhere else. This is the motivation behind Faber polynomials and series. Recall that function spaces are vector spaces whose elements are functions, and the space is endowed with a norm; a way of measuring distance between functions. With some more technicalities, we can ``complete" the space so that we may do analysis as we know it within the space. Note that these spaces are infinite-dimensional, and have differing behaviour to their finite-dimensional counterparts. We can think of a power series of a function as a representation of said function with respect to a basis of polynomials of the form $(z-a)^n$, where $a$ is a point in $\Omega$. While useful, a regular power series representation of a function does not capture the true domain of analyticity in general. The idea behind Faber polynomials is that they provide a basis for the Dirichlet space $\mathcal{D}_{\infty}(\Omega)$ that is formed to fit the ``shape" of $\Omega$. More specifically, they give us to way to write a function in the space as a series of polynomials that converges on $\Omega$. Furthermore, this representation can be shown to be unique, and in fact existence and uniqueness occur if and only if the boundary of $\Omega$ is a quasicircle. Deriving the Faber polynomials can be done in many different ways, but the way it is done in \cite{tietz1957faber} actually comes from a version of the Jump formula (see Appendix \ref{Appendix 2} for the details of this application).

\section{Cauchy Integrals in Dirichlet Space}

The importance of the Cauchy integral in complex analysis is clear. An idea that has been looming in the background\textemdash allowing us to prove many theorems\textemdash is the \emph{regularity} of the contours of integration. For simplicity, we have assumed our contours to be smooth, that is, continuously differentiable, simple, and having non-vanishing tangent vectors. Moreover, many of the standard complex analysis results found in an introductory text also assume such a level of regularity. As previously discussed, quasicircles are so irregular that we don't even have a working definition for integration on them. In this section, we will present a highlight reel of results involving Cauchy integrals that have been attained without a proper definition.

\smallskip

Let $\mathcal{H}(\Gamma, \Omega)$ denote the set of boundary values for elements in the harmonic Dirichlet space of a Jordan domain $\Omega$. Recall that in the homotopic argument used in the first chapter to prove the Cauchy integral formula, we invoked the Riemann mapping theorem to map concentric circles in the unit disk to smooth closed curves in the interior of a contour. We want to make used of this procedure for Jordan domains. For a given harmonic function $h$ in $\mathcal{H}(\Gamma, \Omega)$, we define the Cauchy integral to be: 
\smallskip
\[
\frac{1}{2\pi i}\int_{\Gamma} \frac{h(\tau)}{\tau - z}\,d\tau = \lim_{r \to 1^{-}} \frac{1}{2\pi i}\int_{f(|\tau| = r)} \frac{H(\tau)}{\tau - z}\,d\tau
\]

\noindent
where $H$ is the unique element in $\dharm(\Omega)$ with boundary values $h$, $z$ is a point not on the curve $\Gamma$, and $f:\mathbb{D} \to \Omega$ is a biholomorphic map from the disk to the domain, that exists by the Riemann mapping theorem. Let $\Omega^+$ and $\Omega^-$ denote the bounded and unbounded components of a Jordan curve $\Gamma$, respectively. Define the integral operator $J_{\Gamma}: \dharm(\Omega^{+}) \to \mathcal{D}_{\infty}(\Omega^{+}\!\cup \Omega^{-})$ by
\smallskip
\[
h \mapsto \lim_{r \to 1^-} \frac{1}{2\pi i} \int_{f(|\tau| = r)} \frac{h(\tau)}{\tau - z}\,d\tau
\]

\noindent
which sends a harmonic function $h$ into the Cauchy integral. Since the Riemann map is analytic, the curves approximating the boundary will be analytic curves, and thus the integral is well-defined. For any Jordan domain, the limit will exist and the map itself is bounded, meaning there is a $c > 0$ with the property that
\smallskip
\[
\| J_{\Gamma}\,h\|_{\mathcal{D}_{\infty}(\Omega^{+}\cup\,\Omega^{-})} \leq c\, \| h \|_{\dharm(\Omega_{+})}
\]

\noindent
for every function $h\in\dharm(\Omega)$. In the theory of functional analysis, boundedness in maps is one of the most important and desirable properties. 

A quick side-note: The codomain of $J_{\Gamma}$ can be written as the direct sum $\mathcal{D}(\Omega^+)\oplus\mathcal{D}_{\infty}(\Omega^-)$. This is done for convenience, as the output of the map depends on which complement of the curve $z$ lies in.

It can be shown that $J_{\Gamma}$ is an isomorphism\textemdash meaning it is bounded and bijective\textemdash if and only if $\Gamma$ is quasicircle. In turn, this is in fact equivalent to existence of a unique, convergent Faber series representation for every function $h \in D_{\infty}(\Omega^-)$.

\smallskip

Let us give a brief overview of a special type of function. An \emph{antiholomorphic} function $\bar{f}$ is simply the complex conjugate of a holomorphic function $f$. These functions are similar to their holomorphic counterparts in that they are angle-preserving, but contrary to holomorphic functions, they are \emph{orientation reversing}. It also turns out that every complex harmonic function can be locally written as the sum of a holomorphic and antiholomorphic function. For our purposes, they are important because the harmonic Dirichlet space can be written as the direct sum of the Dirichlet space and the conjugate space normalized at 0:
\smallskip
\[
\dharm(\Omega^+)=\mathcal{D}(\Omega^+)\oplus\overline{\mathcal{D}_{0}(\Omega^+)}.
\]

\noindent
Suppose now that $0 \in \Omega^+$ and $\infty \in \Omega^-$. Define the integral operator $J_{+-}: \overline{\mathcal{D}_{0}(\Omega^{+})} \to \mathcal{D}_{\infty} (\Omega^-)$ by
\smallskip
\[
\bar{h} \mapsto \lim_{r \to 1^-} \frac{1}{2\pi i} \int_{f(|\tau| = r)} \frac{h(\tau)}{\tau - z}\,d\tau
\]

\noindent
which takes the harmonic conjugate of $h$ and applies the Cauchy integral to it. In this case, $z\in\Omega^-$. There is a theorem that shows the map $J_{+-}$ is an isomorphism if and only if $\Gamma$ is a quasicircle. That is, given $\varphi$ as the boundary values of a harmonic function in $\dharm(\Omega^+)$, the surjectivity of the map means that $\varphi$ is the difference of analytic functions in $\mathcal{D}_{\infty} (\Omega^+)$ and $\mathcal{D}_{\infty} (\Omega^-)$, respectively, and injectivity of the map guarantees that the decomposition is unique. The similarity to the Jump problem is no coincidence.

\smallskip

To recap, all of these maps involving the Cauchy integral have what analyists consider “desirable” properties only if the curve is a quasicircle.

\smallskip

Now, let $f$ be a Riemann map from $\Omega^+$ onto the unit disk $\mathbb{D}$. The map $C_{f^{-1}}: \overline{\mathcal{D}_0(\mathbb{D})} \to \mathcal{D}_p(\Omega^+)$ defined by 
\smallskip
\[
h \mapsto h(f^{-1})
\]

\noindent
is an isometry, meaning that is norm-preserving and an isomorphism. As a corollary, it can be shown that the composition $I_f \coloneqq J_{+-}\circ C_{f^{-1}}: \overline{\mathcal{D}_0 (\mathbb{D})}\to\mathcal{D}_{\infty}(\Omega^-)$ is an isomorphism. This particular map is important because it can be used to obtain the Faber polynomials of the domain. First, the monomial basis under the map becomes the Faber polynomials of the set $\Omega^+$, that is,
\smallskip
\[
I_f(\bar{z}^n) = F_n
\]

\noindent
where $F_n$ is the $n$th Faber polynomial of $\Omega^-$. For a function $\overline{h(z)} = \sum_{n=1}^{\infty} a_n \bar{z}^n$ converging in the space  $\overline{\mathcal{D}_{0}(\mathbb{D})}$, we can use the map $I_f$ to obtain a Faber series of an element in $\mathcal{D}_{\infty}(\Omega^-)$:
\smallskip
\[
I_f(\bar{h}) = \sum_{n=1}^{\infty} a_n F_n .
\]

The existence of the series follows from surjectivity of $I_f$, and uniqueness from injectivity. The convergence of this Faber series follows from boundedness of $I_f$ plus convergence of $\bar{h}$. It is also the case that every function in $\mathcal{D}_{\infty}(\Omega^-)$ has a unique Faber series converging in the space if and only if $\Gamma$ is a quasicircle. Combining these observations and the fact that $I_f$ is easily invertible, we can use it to derive the Faber series of any element in $\mathcal{D}_{\infty}(\Omega^-)$. A similar procedure can be used for the Dirichlet space of the interior domain.

\smallskip

The point is not to give the reader a deep understanding of the result presented here. Rather, it is to show than quasicircles, the Jump problem, Dirichlet spaces, and Faber series are all connected, and in particular to show that quasicircles are the natural class of curves for these objects.

\section{The Problem}

In this chapter, we talked about generalizing the Cauchy integral and Jump problem, and some of the strides made to clarify these objects in a greater context. Remarkably, all the work done by mathematicians in the field has done so without using a principal value definition for the Cauchy integral. The question we ask is:

\bigskip
\centerline{\emph{How do we define a principal value integral on a quasicircle?}}
\bigskip

As of the writing of this article, it is unknown if it is even possible to do so. Quasicircles are so poorly-behaved, that even between any two distinct points located on one, the path along the curve connecting the two points may have \emph{infinite} length. To gain an initial understanding of why this is an issue, consider the definition of the Riemann integral \textemdash the first integral encountered by a student of analysis. When looking at the terms in a Riemann sum, we are multiplying the function value at a tag-choice by the \emph{length} of an interval. More general definitions of integration are not a remedy, either. Thus, integrating over a quasicircle in general doesn't even make sense. This is an odd state of affairs, seeing as quasicircles are a fundamental object in complex analysis, and integration is one of the fundamental ways we \emph{do} analysis.

\smallskip

Solving this particular problem is also important for the new connections it may grant between existing objects\textemdash let alone new, fruitful mathematics is could possibly create in the process. First, the Jump problem for quasicircles is (in a sense) incomplete without a proper definition, so we only have part of the picture. Historically, the Jump problem on well-behaved curves was used to solve the Dirichlet problem. Building the understanding to form this relation impacted many other important analytic objects\textemdash especially integral operators. These days, it is the Dirichlet problem (along with other techniques) that is used to understand the Jump formula on quasicircles. Going in the reverse direction would surely lead to a much better understanding of the tools involved, although the full extent of its impact is hard to predict.

As is often the case with complex analytic objects, if we can solve the complex version of a problem, it can help clarify our picture of the real version of the problem. The same is true here. There are two steps for using the real version of Jump problem to solve the Dirichlet (and Neumann) boundary value problems. First, you solve an integral equation on the boundary of the domain that involves principal value integral. The solution to the equation is a function. With this function you can use other integrals to obtain the solution. It is impossible to do the first step without first having a principal value integral, so there is no way to proceed in the case that the boundary is a quasicircle. 

\begin{appendices}
\chapter{The Jump Formula For Analytic Functions}
\label{Appendix 0}

The use of smooth contours and Hölder continuous functions to develop the Jump formula in Section \ref{theJumpFormula} was not by accident. These particular regularity conditions were crucial for that specific argument to work. It is quite natural to ask about how we may weaken these conditions to obtain Jump formulas for more general objects. Analysts do this sort of thing all the time; figuring out what essential properties are needed for a theorem to be true, and looking for counterexamples after removing too many conditions. Over time, we are able to collectively converge on what is the ``best" and ``correct" theorem. 

Achieving a more substantial result than the one we landed on would require far more advanced techniques, so we are not going to do exactly that here. Instead, we are going to look at how the argument changes if we simply \emph{alter} the assumptions. More specifically, we will require that the density function is analytic on the entire contour. Analyticity is so strong that we can essentially drop all but the most reasonable conditions for the contour. While not nearly as interesting of a result developed throughout chapter 1, the proof is simple, and the result of the Jump formula with these stronger conditions helps to motivate our investigation of the stronger theorems from chapter \ref{chapter1}.

\begin{theorem}[Jump Formula for Closed Rectifiable Curves]
Let $C$ be a rectifiable simple closed contour in the $\tau$-plane, and let $\varphi: C \to \mathbb{C}$ be a function that is analytic on $C$. If $\Phi^{\pm}: D^{\pm} \to \mathbb{C}$ is the Cauchy type integral over $C$ with density $\varphi$, then both $\Phi^+$ and $\Phi^-$ can be analytically continued to $C$, and the identity
\smallskip
\[
\varphi(\tau) = \Phi^+(\tau) - \Phi^-(\tau)
\]

\noindent
holds for every $\tau$ on $C$.
\end{theorem}

\begin{proof}
As $\varphi$ is analytic on $C$, in particular it is analytic on some open set containing $C$, which in general forms a doubly-connected region. Within this region, we define two new rectifiable simple closed contour $C^+$ and $C^-$ that lie within $D^+$ and $D^-$, respectively. Let $\tau_0$ be a point on $C$. For $z \in D^{+}$, by the principle of deformation of paths we have
\smallskip
\[
\Phi^+(z) = \frac{1}{2\pi i}\int_{C^-} \frac{\varphi(\tau)}{\tau - z} \,d\tau
\]

The right hand side is analytic on the interior of $C^-$, and equals $\Phi^+$ when restricted to $D^+$. Therefore, $\Phi^+$ is analytic on $D^+ \cup C$. Since  $\tau_0$ lies in this region, it follows that $\Phi^+$ is continuous at $\tau_0$ on approach within $D^+$. Similarly, we can show $\Phi^-$ is analytic on $D^- \cup C$ using $C^+$, so it too is continuous at $\tau_0$, this time on approach from within $D^-$.

Next, connect $C^+$ and $C^-$ with a straight line segment $L$ that does not pass through $\tau_0$. We are going to integrate starting at $L$, clockwise around $C^+$, up $L$ to $C^-$, counterclockwise around $C^-$ and then back down $L$. Integrating over $L$ both ways contributes nothing to the integral, and creates an integration path that completely surrounds $\tau_0$. Letting $\gamma \coloneqq -C^+ + L + C^- - L$ and applying the Cauchy integral formula, we compute:
\smallskip
\begin{align*}
\lim_{\substack{z \to \tau_0 \\ z \in D^+}} \Phi^+(z) - \lim_{\substack{z \to \tau_0 \\ z \in D^-}} \Phi^-(z) &= \Phi^+(\tau_0)-\Phi^-(\tau_0) \\ &= \frac{1}{2\pi i}\int_{\gamma} \frac{\varphi(\tau)}{\tau - \tau_0} \,d\tau \\
&= \varphi(\tau_0).
\end{align*}
\end{proof}

\begin{remark}
The reader may wish to consult Apostol's book \cite{apostol1957mathematical} to verify that the basic contour integral properties used in this proof still hold true when dealing with rectifiable curves (rather than smooth contours). 
\end{remark}

\chapter{Jordan Curves and Open Sets in S\textsuperscript{1}}
\label{Appendix 1}

Here we present a proof that for every point on a smooth contour, we can find a neighbourhood so small as to only contain a single connected component, or ``arc", of that contour. We mentioned this property in the remark following the definition of the singular value integral in Section \ref{Principal Value Integral}, and it is a highly desirable property to accompany the integral. In fact, this property holds for a far more general class of curves in the plane, known as \emph{Jordan curves}. This is an instances where a generalization makes the argument far cleaner and easier to understand. We just need the right definition. 

This appendix requires some basic elements of point-set topology, which are not assumed prerequisites of the main article. For valuable references, see \cite{kalajdzievski2015illustrated}, \cite{munkrestopology}, and \cite{newman1939elements}. Taking this detour is not strictly necessary to understand the principal value integral and Jump problem, and can be skipped at the discretion of the reader. However, it has been included here in this appendix for completeness, and relies on some concepts that are important for readers who continue to study this field.

We start by recalling the famous representation theorem for open subsets of the real line.

\begin{theorem}[Representation of Open Sets in $\mathbb{R}$]
\label{Characterization of open sets in R}
Every non-empty open set in $\mathbb{R}$ can be written uniquely as the union of countably many disjoint connected components. 
\end{theorem}

\noindent
A proof can be found in Apostol's Mathematical Analysis \cite{apostol1957mathematical}. It turns out the same type of representation is true for open sets in the unit circle $S^1$. From here on, assume that subsets of $\mathbb{R}$ and $\mathbb{C}$ are automatically equipped with the subspace topology. It is easy to see that the basic open sets of $S^1$ are of the form $\{ e^{it} : a<t<b \}$, where $a$ and $b$ are real numbers. We call sets of this form \emph{open intervals} of $S^1$, the same way we do for basic open sets in $\mathbb{R}$.

\begin{theorem}[Representation of Open Sets in $S^1$]
\label{Characterization of open sets in s1}
Every open set in $S^1$ can be written uniquely as the union of countably-many pairwise disjoint connected components in $S^1$.
\end{theorem}

\begin{proof}
Let $U$ be an open subset of $S^1$. If $U$ is the entire space, there is nothing to prove, so let us assume henceforth that $U \neq S^1$. Define a function $f:\mathbb{R} \to S^1$ by
\smallskip
\[
f(t) = e^{it}
\]

\noindent
and choose a point $t_0$ such that $f(t_0)=e^{it_0}$ is not in $U$. It is clear that the restriction $f\vert_{(t_0,t_0 + 2\pi)}$ is a bijection. By continuity of $f$, the set $f^{-1}(U) \cap (t_0,t_0 + 2\pi)$ is open in $\mathbb{R}$, so by Theorem \ref{Characterization of open sets in R}, we can uniquely write
\smallskip
\[
f^{-1}(U) \cap (t_0, t_0 + 2\pi) = \bigcup_{n \in \mathbb{N}} I_n
\]

\noindent
where the $I_n$'s are pairwise disjoint connected components in $\mathbb{R}$. Furthermore, each $I_n$ is the preimage of an open interval in $S^1$ under $f$, and each open interval in $U$ will have a distinct preimage in $(t_0, t_0 + 2\pi)$ under $f^{-1}$. What we have done is create a one-to-one correspondence between the connected components in $\mathbb{R}$ and $S^1$, and the result follows.
\end{proof}

As a point of comparison, let us recall the usual definition of simple closed contour (as found in \cite{brown2009complex}, for example).

\begin{definition}[Simple Closed Contour]
\label{Simple Closed Contour}
A simple closed contour $C$ is the image of a continuous function $\alpha: [a,b] \to \mathbb{C}$ that is injective on $(a,b)$, and $\alpha(a) = \alpha(b)$. 
\end{definition}

Next, we give the topological definition of a simple closed contour in $\mathbb{C}$ (similar to the one found in \cite{newman1939elements}).

\begin{definition}[Jordan Curve]
\label{Jordan Curve}
A Jordan curve is the image of a continuous map $\alpha: S^1 \to \mathbb{C}$ that is a homeomorphism onto its image.
\end{definition}

Where smooth contours are nice and relatively easy to work with, Jordan curves are the absolute worst class of plane curves, even more so than the quasicircle (see Section \ref{quasicircle}). For example, in Pommerenke's book on Conformal maps \cite{pommerenke2013boundary}, some properties that characterize quasicircles are explored, and one of these properties is that they cannot have cusps. Jordan curves certainly can, which introduces far more complexity to their analysis.

With that being said, the result we are after is truly topological in nature, so it should be expected that the argument will be more natural using the Jordan curve. Technically, we do need to show that this topological definition of a Jordan curve lines up with the typical analytic definition of simple closed curve. We summarize this with the following proposition.

\begin{proposition}
The definitions of Jordan curve and simple closed curve are equivalent.
\end{proposition}
\begin{proof}
Let $f: [0,1] \to S^1$ denote the following parameterization of the unit circle:
\smallskip
\[
f(t) = e^{i2\pi t}.
\]

First, suppose $J$ is a Jordan curve and $\alpha: S^1 \to J$ a homeomorphism. Then $\alpha \circ f: [0,1] \to J$ is a continuous parameterization of $J$ that is injective except at its endpoints, so $J$ is a simple closed contour.

Now suppose that $C$ be a simple closed contour, and $\beta: [0,1] \to C$ is a parameterization of $C$. For any closed subset of $[0,1]$, its image under $\beta$ will be compact, and hence closed since $C$ is a Hausdorff space (Page 161 in \cite{kalajdzievski2015illustrated}), so $\beta$ is a closed map. Define an equivalence relation $\sim_{\beta}$ on $[0,1]$ by $a\sim_{\beta} b$ if and only if $\beta (a)=\beta (b)$. It follows that $[0,1]/\!\sim_{\beta}$ is homeomorphic to $C$ (Page 77 in \cite{kalajdzievski2015illustrated}). We can use the same argument with the map $f$ to show that $[0,1]/\!\sim_{f}$ is homeomorphic to $S^1$. Therefore, $C$ is a a Jordan curve.
\end{proof}

\begin{corollary}
Every closed smooth contour is a Jordan curve.
\end{corollary}

In the case that a smooth contour is open, we can simply attach another simple contour to its endpoints in a way that doesn't intersect the original curve anywhere else to create a Jordan curve. We are now equipped to prove the main result of this section.

\begin{theorem}
Let $\Gamma$ be a Jordan curve in the complex plane. For every $\tau_0$ on $\Gamma$, there is an open neighbourhood of $\tau$ that contains exactly one connected component of $\Gamma$.
\end{theorem}

\begin{proof}
Let $f: S^1 \to \Gamma$ be a homeomorphism. Choose $r>0$ small enough so that the open ball $B(\tau_0,r)$ excludes at least one point of $\Gamma$. Consider the open set $f^{-1}(B(\tau_0,r)\cap \Gamma)$ in $S^1$. By Theorem \ref{Characterization of open sets in s1}, we can write 
\smallskip
\[
f^{-1}(B(\tau,r)\cap \Gamma) = \bigcup_{n\in\mathbb{N}} I_n
\]

\noindent
where the $I_n$'s are pairwise disjoint connected components of $S^1$. Let $t_0 = f^{-1}(\tau_0)$. If necessary, relabel the intervals so that $I_1$ is the unique component containing $t_0$. Since $f$ is a homeomorphism, it follows that $f(I_n)$ is an open connected component of $\Gamma$ for every $n\in\mathbb{N}$, and the $f(I_n)$'s are pairwise-disjoint. In particular, we have that there is an open set $V\subset\mathbb{C}$ such that $f(I_1) = \Gamma\cap V$, and since $f(I_i) \cap f(I_j) = \varnothing$ whenever $i \neq j$, it follows that $f(I_n) \cap V = \varnothing$ for every $n>1$. Choosing $0<\delta<r$ small enough so that $B(\tau_0, \delta) \subset V$, we have that $\Gamma \cap B(\tau_0,\delta)$ contains exactly one connected component, which completes the proof.
\end{proof}
\chapter{Existence of Faber Series}
\label{Appendix 2}

In Section \ref{FaberPolynomialDiscussion}, we introduced the idea of a type of power series that is fit for a specific domain, in the sense that we can find a set of polynomials that capture the ``shape" of a domain, and form a sort of ``basis" for series representation of analytic functions that will converge everywhere on that domain (rather than simply on open disks). Following a paper by H. Tietz \cite{tietz1957faber}, we are going to apply the Jump formula in an effort prove the existence of such polynomials for a special class of domain, and in turn a series representation (called a \emph{Faber series}) for analytic functions within said domain.

The theorem presented here is both important to the area of research discussed in chapter 2, and a very nice and relatively simple application of far more complicated ideas that are none-the-less essential for the reader pursuing further study in the field. Unlike the main text, this appendix will not develop this theory. Rather, we point the reader towards references that fill in the gaps. At the very least, the reader will benefit from familiarizing themselves with the contents of Lang's book \cite{lang2013complex}, specifically the second chapter on formal power series.

\smallskip

To begin, let us introduce some essential terminology.

\begin{definition} [Conformal Map]
Let $D$ be a domain in the complex plane. We say that a function $f: D \to \mathbb{C}$ is a conformal map if it is both analytic and injective.
\end{definition}

\begin{definition} [Pole at Infinity]
For an analytic function $f$, we say $f$ has a pole of order n at infinity if $f(\frac{1}{z})$ has a pole of order n at 0.
\end{definition}

The following characterizes entire functions with a pole at infinity. The proof is short and is left to the reader to verify.

\begin{proposition}
\label{pole at infinity iff polynomial}
An entire function $p(z)$ has pole of order $n$ at infinity if and only if $p(z)$ is a polynomial of degree n.
\end{proposition}

We are only going to look at domains whose boundaries are especially regular. This new class of contour will do just that.

\begin{definition} [Analytic Jordan Curve]
A simple closed contour $\gamma$ is said to be analytic if there exists $r>1$ and a conformal map $\alpha$ on the annulus $\{ z: \frac{1}{r} < |z| < r\}$ such that $\alpha\vert_{|z|=1} = \gamma$.
\end{definition}

For a domain $D$ bounded by an analytic Jordan curve, we say that $D$ is an \emph{analytic Jordan domain}.

\smallskip

The main tool used in the proof is the \emph{Riemann Mapping Theorem}. This theorem is a cornerstone of geometric function theory. We simply state it for reference, as the proof is not easy and can be found in many places (see \cite{lang2013complex} for example).

\begin{theorem}[Riemann Mapping Theorem]
\label{Riemann mapping theorem}
If $U$ is a simply connected subset of the complex plane that is not the entire plane, then there exists a conformal map $f: U \to \mathbb{D}$.
\end{theorem}

It can also be shown that the function $f$ in the Riemann mapping theorem has a conformal inverse. We can now state and prove the desired theorem regarding Faber series. 

\begin{theorem}[Existence of Faber Series on Analytic Jordan Domains]
Let $C$ be an analytic Jordan curve. For any function $f$ that is analytic on $D^+ \cup C$, it can be represented in $D^+$ by a convergent series of the form
\smallskip
\[
f(z) = \sum_{n=0}^{\infty} a_n \Psi_n (z)
\]

\noindent
where $\Psi_n$ is a degree-n polynomial known as the n-th Faber polynomial of $D^+$.
\end{theorem}

\begin{proof}
By the Riemann mapping theorem, there exists a conformal map $w = g(z)$ from $D^-$ onto the exterior of the unit circle $S^1$, which we will denote by $\mathbb{D}^{-1}$. This is because $z \mapsto 1/z$ is itself a conformal map from $\mathbb{D}$ onto $\mathbb{D}^{-1}$, using the convention that $0 \mapsto \infty$. Now, we make use of a result generalizing the Schwarz reflection principle that allows us to analytically continue conformal maps to a neighbourhood of their boundary when the boundary of both the domain and image are analytic Jordan curves. For reference, see Proposition 1.3 in the book \cite{shapiro1992schwarz}. Thus, $g$ can be analytically continued to a neighbourhood of $C$. It follows that $f^{-1}$ is analytic on a neighbourhood of $\overline{\mathbb{D}}^{-1}$, and hence $f \circ g^{-1}$ is analytic on some neighbourhood of $\mathbb{D}^{-1}$. We then know that $f \circ g^{-1}$ can be represented with a Laurent expansion on $S^1$:
\smallskip
\[
f \circ g^{-1}(w) = \sum_{n=-\infty}^{\infty} a_n w^n .
\]

\noindent
Note that the Laurent coefficients $a_i$ are uniquely determined by $f \circ g^{-1}$. Replacing $w$ with $g$, we obtain the following expression for $f$ on $C$:
\smallskip
\[
f(z) = \sum_{n=-\infty}^{\infty} a_n [g(z)]^n.
\]

\noindent
From the Cauchy integral formula, we have $f(z) = \Phi^{+}(z)$ for $z\in D^+$. Integrating term-by-term, we can rewrite the above series for $f$ as
\smallskip
\[
f(z) = \Phi^+(z) = \sum_{n=-\infty}^{\infty} a_n L^+(g(z)^n) \tag{1}
\]

\noindent
now converging on $D^+$, and using the notation
\smallskip
\[
L^{\pm}(\varphi(z)) \coloneqq \frac{1}{2\pi i}  \int_C \frac{\varphi(\tau)}{\tau - z} \,d\tau
\]

\noindent
to denote the Cauchy type integral on $D^{\pm}$ (respectively) for a particular density function $\varphi$. Next, we calculate the Laurent series of $g$ in $D^-$. As $g$ is injective near $\infty$, we must have $g(\infty)=\infty$, and thus the function $H(z) \coloneqq \left[g(1/z)\right]^{-1}$ is injective near 0, implying that $H'(0) \neq 0$. Moreover, note that $H(0) = 0$. All this to say that $H$ has a power series expansion at 0 of the form
\smallskip
\[
H(z) = \sum_{n=1}^{\infty} b_n z^n
\]

\noindent
where $b_1$ is nonzero. We now manipulate $g$ as a formal power series using $H$:
\smallskip
\begin{align*}
g(z) = \frac{1}{H\left(\frac{1}{z}\right)} &= \frac{1}{\frac{b_1}{z} + \frac{b_2}{z^2} + \cdots} \\
&= \frac{z}{b_1}\cdot\frac{1}{1+ \frac{b_2}{b_1 z} + \frac{b_3}{b_2 z} + \cdots} \\
&= \frac{z}{b_1} \cdot \left[ 1 - \left( \frac{b_2}{b_1 z} + \frac{b_3}{b_2 z} + \cdots\right) + \left( \frac{b_2}{b_1 z} + \frac{b_3}{b_2 z} + \cdots\right)^2 - \cdots \right].
\end{align*}

\noindent
Multiplying this expression out, we get the following representation for $g$:
\smallskip
\[
g(z) = c_1 z + c_0 + \frac{c_{-1}}{z} + \frac{c_{-2}}{z^2} + \cdots
\]

\noindent
for some complex coefficients $c_i$, $i \leq 1$, and where $c_1$ is nonzero. This works because $g$ is analytic on a neighbourhood of infinity, so considering $\labs z \rabs $ ``large enough", we can switch from the formal power series to a convergent one in the typical analytic sense by uniqueness of power series.

Now we calculate the power series expansion of $g(z)^n$ for $n>0$. Let $\omega\coloneqq 1/z$. Using formal power series manipulation, $\text{ord}\left[ g(1/\omega)^n\right] = n(\text{ord}\left[ g(1/\omega)\right]) = -n$, which means $g(1/\omega)^n$ takes the form
\smallskip
\[
g\left(\frac{1}{\omega}\right)^n = \frac{d_n}{\omega^n} + \cdots + \frac{d_1}{\omega} + d_0 + d_{-1}\omega + \cdots 
\]

\noindent
for some complex coefficients $d_i$, $i \leq n$, and where $d_n$ is nonzero. Changing variables, we can write
\smallskip
\[
g(z)^n = d_n z^n + \cdots + d_1 z + d_0 + \frac{d_{-1}}{z} + \cdots.
\]

\noindent
Next, for every integer $n$, the Jump formula says that the identity
\smallskip
\[
g^n = L^+(g^n) - L^-(g^n) 
\]

\noindent
hold on $C$. Since each side of the equation is analytic on an open neighbourhood of $C$, it follows that the relation will also hold on a neighbourhood of $C$. As $g^n$ and $L^-(g^n)$ are analytic on $D^-$, and we have equality on an open set, we can analytically continue $L^+(g^n)$ to the entire plane so that
\smallskip
\[
L^+(g^n) = g^n + L^-(g^n) \tag{2}
\]
\noindent
Recall from Theorem \ref{CTI vanish at infinity} that $L^-(g^n)$ vanishes as infinity. For $n<0$, $g(\infty)^n = 0$, so from (2) $L^+(g^n)$ vanishes at infinity, and in particular doesn't have a pole at infinity. Thus it is must be constant, forcing $L^+(g^n) \equiv 0$.

When $n\geq 0$, using Proposition \ref{pole at infinity iff polynomial}, $L^+(g^n)$ must be a polynomial of degree $n$ (the sought after Faber polynomial associated to the domain $D^+$), which we denote by $\Psi_n$. Thus we can write (1) as
\smallskip
\[
f(z) = \sum_{n=0}^{\infty} a_n \Psi_n(z)
\]

\noindent
for every $z\in D^+$.
\end{proof}

\smallskip

\begin{remark}
A far messier, but still viable approach to proving that $L^+(g^n)$ is a polynomial for $n \geq 0$, and vanishes for $n < 0$ can be done by continuing the approach of formal power series manipulation. 

First, we show that $L^+(g(z)^n)$ is an $n$-degree polynomial for $n\geq0$. By analyticity of the integrand, we can move the contour $C$ out to a circle $C'$ satisfying $\labs \zeta \rabs > \labs z \rabs$ for every $\zeta \in C'$. Expanding the kernel, we have
\smallskip
\[
\frac{1}{\zeta - z} = \frac{1}{\zeta} \cdot \frac{1}{1-\frac{z}{\zeta}} = \frac{1}{\zeta} + \frac{z}{\zeta^2} + \frac{z^2}{\zeta^3} + \cdots
\]

\noindent
for $\labs \zeta \rabs > \labs z \rabs$. Now multiplying by the density $g^n$:
\smallskip
\[
\frac{g(\zeta)^n}{\zeta - z} = \left( \frac{1}{\zeta} + \frac{z}{\zeta^2} + \frac{z^2}{\zeta^3} + \cdots \right) \cdot \left( d_n \zeta^n + \cdots + d_1 \zeta + d_0 + \frac{d_{-1}}{\zeta} + \cdots \right).
\]

\noindent
Distributing on the left hand hand and integrating term by term, it follows that
\smallskip
\[
L^+(g(z)^n) = \frac{1}{2\pi i}\int_{C'} \frac{g(\zeta)^n}{\zeta - z} \,d\tau = d_n z^n + \cdots + d_1 z + d_0
\]

\noindent
using the fact that for every integer $n$:
\smallskip
\[
\int_{C'} \tau^n \,d\tau = \begin{cases}
2\pi i &\! \textup{if} \;\: n = -1\\
0 &\! \textup{if} \;\: n \neq -1 .
\end{cases}
\]

Second, we show that $L^+(g(z)^n) \equiv 0$ for every $n<0$. Fix $n>0$, and write
\begin{align*}
\left[ g(z)^n \right]^{-1} = \frac{1}{g(z)^n} &= \frac{1}{d_n z^n + \cdots + d_1 z + d_0 + \frac{d_{-1}}{z} + \cdots} \\
&= \frac{1}{d_n z^n} \cdot \frac{1}{1 + \cdots + \frac{d_0}{d_n z^n} + \frac{d_{-1}}{d_n z^{n+1}} + \cdots}\\
&= \frac{1}{d_n z^n} \cdot \left[ 1 - (\text{powers of 1/z}) +  (\text{powers of 1/z})^2 - \cdots \right]\\
&= \frac{\omega_{n}}{z^n} + \frac{\omega_{n+1}}{z^{n+1}} + \cdots
\end{align*}

\noindent
for some coefficients $\omega_i$, $i\geq n$, and $\omega_n \neq 0$. Thus when $n<0$, another computation involving the kernel as done above yields
\[
L^+(g(z)^n) = \frac{1}{2\pi i}\int_{C'} \frac{g(\zeta)^n}{\zeta - z} \,d\tau = 0.
\]
\end{remark}
\end{appendices}

\nocite{*}
\printbibliography[title = {References},heading=bibintoc]

@book{gakhov2014boundary,
  title={Boundary value problems},
  author={Gakhov, F. D.},
  year={2014},
  publisher={Elsevier}
}

@book{brown2009complex,
  title={Complex variables and applications eighth edition},
  author={Brown, James W. and Churchill, Ruel V},
  year={2009},
  publisher={McGraw-Hill Book Company}
}

@book{pommerenke2013boundary,
  title={Boundary behaviour of conformal maps},
  author={Pommerenke, Christian},
  volume={299},
  year={2013},
  publisher={Springer Science \& Business Media}
}

@book{lang2013complex,
  title={Complex Analysis},
  author={Serge Lang},
  volume={103},
  year={2013},
  publisher={Springer; Fourth Edition}
}

@misc{apostol1957mathematical,
  title={Mathematical Analysis: Modern Approach to Advanced Calculus second edition},
  author={Tom M. Apostol},
  year={1974},
  publisher={Pearson}
}

@book{muskhelishvili2008singular,
  title={Singular integral equations: boundary problems of function theory and their application to mathematical physics},
  author={Nikola Muskhelishvili},
  year={2008},
  publisher={Courier Corporation}
}

@book{saxe2002beginning,
  title={Beginning functional analysis},
  author={Karen Saxe},
  year={2002},
  publisher={Springer}
}

@article{tietz1957faber,
  title={Faber series and the Laurent decomposition.},
  author={H. Tietz},
  journal={Michigan Mathematical Journal},
  volume={4},
  number={2},
  pages={175--179},
  year={1957},
  publisher={University of Michigan, Department of Mathematics}
}

@book{gelbaum2003counterexamples,
  title={Counterexamples in Analysis},
  author={Gelbaum, Bernard R. and Olmsted, John M.H.},
  year={2003},
  publisher={Dover Publications}
}

@misc{schippersConformalMaps,
  title={Lecture Notes on Conformal Mappings},
  author={Eric Schippers},
}

@book{munkrestopology,
  title={Topology},
  author={James R. Munkres},
  year={2017},
  publisher={Pearson; Second Edition}
}

@book{kalajdzievski2015illustrated,
  title={An illustrated introduction to topology and homotopy},
  author={Kalajdzievski, Sasho},
  year={2015},
  publisher={CRC Press}
}

@book{newman1939elements,
  title={Elements of the topology of plane sets of points},
  author={Newman, M. H. A.},
  year={1939},
  publisher={Cambridge}
}

@book{shapiro1992schwarz,
  title={The Schwarz function and its generalization to higher dimensions},
  author={Shapiro, Harold S},
  volume={4},
  year={1992},
  publisher={John Wiley \& Sons}
}

@book{adamsEssexcalculus,
  title={Calculus: A Complete Course Ninth Edition},
  author={Adams, Robert A. and Essex, Christopher},
  year={2017},
  publisher={Pearson Canada}
}

@book{lehto2012univalent,
  title={Univalent functions and Teichm{\"u}ller spaces},
  author={Lehto, Olli},
  volume={109},
  year={2012},
  publisher={Springer Science \& Business Media}
}

\end{document}